\documentclass[reqno]{amsart}

\usepackage{amsmath}
\usepackage{amssymb}
\usepackage{amsfonts}
\usepackage{mathtools}
\usepackage{bm}
\usepackage{xcolor}
\usepackage{accents}
\usepackage{pgfplots}
\usepackage{fixmath}
\usepackage{relsize}
\usepackage{subfigure}
\usepackage{lipsum}
\usepackage{algorithmic}

\usepackage[noadjust]{cite}

\newtheorem{thm}{Theorem}[section]
\newtheorem{corollary}[thm]{Corollary}
\newtheorem{lemma}[thm]{Lemma}
\newtheorem{prop}[thm]{Proposition}
\newtheorem{rem}[thm]{Remark}
\newtheorem{dfn}[thm]{Definition}

\newtheorem{assumption}[thm]{Assumption}

\newcommand{\RgO}{\ensuremath{\mathbb{R}_{>0}}}
\newcommand{\RgeO}{\ensuremath{\mathbb{R}_{ \geq 0}}}
\newcommand{\RleO}{\ensuremath{\mathbb{R}_{ \le 0}}}

\newcommand{\Rat}[1]{\ensuremath{\mathbb{R}^{#1}}}
\newcommand{\B}{\mathcal{B}}
\newcommand{\bone}{\mathbf{1}}

\newcommand{\D}{\mathcal{D}}

\newcommand{\I}{\mathcal{I}}

\newcommand{\bN}{\mathbb{N}}

\newcommand{\bE}{\ensuremath{\mathbb{E}}}

\renewcommand{\S}{\mathbb {S}}

\newcommand{\supp}[1]{{\rm supp}(#1)}
\newcommand{\cone}[1]{{\rm cone}(#1)}
\newcommand{\longthmtitle}[1]{\mbox{}{\bf \textit{(#1).}}}
\newcommand{\longthmtitlealt}[1]{\mbox{}{\bf \textit{(#1)}}}

\newcommand{\conv}{\ensuremath{\operatorname{conv}}}
\newcommand{\cl}{\ensuremath{\operatorname{cl}}}
\newcommand{\bdry}{\ensuremath{\operatorname{bdry}}}

\newcommand\oprocendsymbol{\hbox{$\bullet$}}
\newcommand\oprocend{\relax\ifmmode\else\unskip\hfill\fi\oprocendsymbol}



\title{Gradient sampling algorithm for subsmooth functions}

\newcommand{\subscr}[2]{#1_{\textup{#2}}}

\newcommand{\dist}{\operatorname{dist}}

\newcommand{\interior}[1]{\operatorname{int}(#1)}

\newcommand{\proj}{\operatorname{proj}}

\newcommand{\eps}{\varepsilon}

\newcommand{\argmin}{\ensuremath{\operatorname{argmin}}}
\newcommand{\argmax}{\ensuremath{\operatorname{argmax}}}

\usepackage{psfrag}

\newcommand{\sign}{\operatorname{sign}}

\renewcommand{\eps}{\ensuremath{\varepsilon}}
\allowdisplaybreaks 

\author[D. Boskos]{Dimitris Boskos}
\address{Delft Center for Systems and Control, TU Delft}
\email{d.boskos@tudelft.nl}

\author[J. Cort\'es]{Jorge Cort\'es}
\address{Department of Mechanical and Aerospace Engineering, UC San Diego}
\email{cortes@ucsd.edu}

\author[S. Mart{\'\i}nez]{Sonia Mart{\'\i}nez}
\address{Department of Mechanical and Aerospace Engineering, UC San Diego}
\email{soniamd@ucsd.edu}

\begin{document}

\maketitle

\begin{abstract}
  This paper considers non-smooth optimization problems where we seek
  to minimize the pointwise maximum of a continuously parameterized
  family of functions. Since the objective function is given as the
  solution to a maximization problem, neither its values nor its
  gradients are available in closed form, which calls for
  approximation. Our approach hinges upon extending the so-called gradient sampling
  algorithm, which approximates the Clarke generalized gradient of the
  objective function at a point by sampling its derivative at nearby
  locations. This allows us to select descent directions around points
  where the function may fail to be differentiable and establish
  algorithm convergence to a stationary point from any initial
  condition. Our key contribution is to prove this convergence by
  alleviating the requirement on continuous differentiability of the
  objective function on an open set of full measure. We further
  provide assumptions under which a desired convex subset of the
  decision space is rendered attractive for the iterates of the
  algorithm.
\end{abstract}

\section{Introduction}

The need to solve optimization problems with non-convex objective
functions is pervasive. A partial list of major applications where one
seeks to minimize non-convex functions includes prediction error
methods in system identification, nonlinear optimal control, and
neural network training. Due to the inherent intractability of
algorithms that search for global optima, in non-convex programming it is typical to settle with stationary points, which satisfy
  first-order necessary local optimality conditions. For smooth
problems, the approach to reach such points rests upon the selection
of a descent direction that is obtained from information about
  the first or second-order derivatives of the objective
  function. When the objective function is no longer differentiable,
finding a descent direction becomes considerably harder and algorithms
rely on subgradient information. While this can provide global
convergence guarantees when minimizing convex functions via
  subgradient methods, even local convergence guarantees are
significantly more challenging in the non-convex case. In
  particular, it is no longer possible to guarantee that a non-smooth minimization
  gets closer to the set of stationary points at each iteration and
  one needs to employ an approximate subdifferential of the function
  at each point to obtain a descent direction instead. In addition to
  this, even when the function is convex but non-smooth, Cauchy's
  algorithm along guaranteed decent directions may zigzag and converge
  to a point that is not a minimum \cite{PW:75}, \cite[Chapter
  VIII.2.2]{JBHU-CL:93}.
As also remarked in~\cite{JVB-FEC-ASL-MLO-LS:20}, the same failure of convergence to a stationary point, or minimum for convex problems, has been established in the 
recent works \cite{JG-ASL:18} and \cite{AA-MLO:20} for non-exact line searches.

To address these challenges, the gradient sampling (GS) algorithm
introduced by Burke, Lewis, and Overton in \cite{JVB-ASL-MLO:05},
approximates the Goldstein $\eps$-subdifferential of the objective function at each iterate by augmenting its gradient information at that point with sampled  gradients at nearby locations where the function is differentiable.  The algorithm then computes the least-norm element from the convex hull of these gradients, which in turn provides a more ``robust" search direction for its next iterate. Our goal in this paper is to extend this algorithm to scenarios where neither the values nor the
gradients of the objective function are available in closed form. This
is the case when the objective function is given as the pointwise
maximum of a family of sufficiently regular functions, and so its
value can only be approximated by solving an inner maximization
problem. In this situation, the almost-everywhere continuous
differentiability of the objective function may fail, which is a
critical condition~\cite{JVB-FEC-ASL-MLO-LS:20} to control the local 
variation of sampled gradients and ensure the algorithm convergence.  
Such classes of objective functions are typically encountered in
non-convex robust optimization
problems~\cite{BH-MD:13,MK-FL-MS:22,DB-ON-KMT:10a,DB-ON-KMT:10b},
including distributionally robust optimization problems with
finite-dimensional distribution parameterizations.  A concrete
instance of the latter problem class is distributionally robust
coverage control optimization, where a non-convex problem of optimal
resource becomes non-smooth due to its robustification.
%

\textit{Literature review:} 
The monographs~\cite{KCK:85,NZS:85,JFB-JCG-CL-CAS:06,YC-JSP:21} provide an
extensive account of methods and problems for non-differentiable,
non-convex optimization.
Typical approaches to such problems include bundle
methods~\cite{HS-JZ:92,WH-CS-MS:16}, derivative-free
algorithms~\cite{AMB-BK-MS:08,JL-MM-SMW:19}, quasi-Newton
methods~\cite{ASL-MLO:13}, smoothing of the objective
function~\cite{XC:12}, and sampling-based
methods~\cite{JVB-FEC-ASL-MLO-LS:20}. A recent strand of research also
considers non-smooth, non-convex optimization algorithms that are
guaranteed to converge in finite time to approximate stationary points
with high probability, and analyses of their
complexity~\cite{MJ-GK-TL-OS-MZ:23,GK-OS:22,DD-DDYTL-SP-GY:22,JZ-HL-SS-AJ:20}.

Nonsmooth optimization problems arise naturally when one seeks to 
to minimize the maximum of a family of functions. This can be manifested as 
a min-max optimization problem, which can in turn be also viewed as a special 
case of bilevel programs or semi-infinite optimization problems
\cite{SD:02,RH-KOK:93,OS:03}. When the objective function of 
the min-max formulation is convex and concave in the minimization and maximization  variables, respectively, then optimal points of the outer minimization 
correspond to saddle points of the min-max problem. Convergence of the saddle 
point dynamics for convex-concave functions can be traced back to the seminal 
treatise \cite{KJA-LH-HZ:58}, while (sub)gradient methods that yield a sequence
converging to a saddle point for this problem class can be found for instance in \cite{DM:77,AN:04,AN-AO:08a}, and the more recent works  \cite{AM-AO-SP:20,RJ-AM:25}.
When the overall objective function is no longer convex-concave,
convergence requirements are typically relaxed to stationary
points. Such a convergence is established in several recent
contributions that rely on gradient descent ascent algorithms
\cite{TL-CJ-MIJ:20,TZ-LZ-AMCA-JB-JL:23,TL-CJ-MIJ:25} and optimistic
gradient descent ascent/extra gradient algorithms
\cite{MB-EYH-AJ:23,PM-YD-HL-MM:22,SL-DK:21,YC-AO-WZ:24}.

When the inner maximization is carried out over a compact set of parameters and the 
function is twice continuously differentiable for each fixed parameter (with derivatives up to second order being jointly continuous with respect to the parameters) the resulting function that gets minimized is said to be of   lower-$\mathcal C^2$ class. The work~\cite{RTR:82} establishes the equivalence between this class and the class of weakly convex functions,  while \cite{DD-DD:19} and \cite{JCD-FR:18} provide stochastic approximation algorithms to minimize the latter. One can also view weakly convex functions as a special case of difference of convex (DC) functions (cf. \cite{JBHU:85}), which form a broad category of costs in the global optimization literature~\cite[Chapter 4]{RH-PMP-NVT:00}.

A relevant class of non-convex optimization problems is related to the problem of optimal coverage, which seeks to optimally place a team of agents across an environment with a spatially distributed quantity of
interest~\cite{JC-SM-TK-FB:02-tra,FB-JC-SM:08-sv}.  
The purpose is to maximize the expected utility of the whole team by making use of a local performance function for each agent. This problem has widespread applications in environmental monitoring and search-and-rescue operations. 
Besides being non-convex, the objective function of the problem typically lacks differentiability at configurations where two agents share the same location. There are also situations where the performance function in the expected cost may be discontinuous, like e.g., when there are hard constraints on the agents' sensing ranges. This generates multiple extra configurations where the expected cost is non-smooth~\cite{JC-FB:09-sirev}. Yet one more source of non-smoothness emerges by accounting for distributional robustness when the density of the quantity of interest is uncertain. This problem has been considered in \cite{DB-JC-SM:23-cdc}, which uses wavelet-based ambiguity sets to capture plausible densities for the quantity of interest~\cite{DB-JC-SM:22-cdc}.

The same problem formulation as in optimal coverage can be considered if one is interested in probability density quantization~\cite{SG-HL:00}. Solutions to the problem for fixed densities typically rely on Lloyd's \cite{QD-ME-LJ:06} or stochastic approximation algorithms ~\cite{GP:98,GP:18}. A significant variant of the same problem is that of optimal data clustering, where the density is replaced by an empirical distribution with a large number of atoms, and the agent positions represent  the centroids of the clusters. The precise goal in this case is to determine the positions of the centroids in order to minimize the average (square) distance of the data from their closest centroid. The works \cite{AMB-RMA-NS:23} and \cite{AMB-NK-ST:25} provide an extensive analysis of nonconvex and non-smooth optimization algorithms to solve this problem. 

Sampling-based methods for non-smooth and non-convex problems
are motivated by the fact that locally Lipschitz functions are 
differentiable almost everywhere and therefore
their Goldstein $\eps$-subdifferential
at a point can be well approximated by
randomly sampling its nearby gradients~\cite{JVB-ASL-MLO:02}. This has
the benefit that the gradients of the function are often much cheaper
to compute compared to its subgradients. To this end, the GS algorithm
introduced in~\cite{JVB-ASL-MLO:05} samples at each iteration 
gradients of the function from a neighborhood of the current point, 
and adjusts this neighborhood appropriately along
the iterations. Then a subsequence of the points 
generated by the algorithm is guaranteed to approach a stationary 
point. There are specific adjustments to the GS algorithm
which account for the issues of checking whether the objective
function is differentiable at its sampled points, and even computing
its derivatives at these points. Specifically, \cite{ESH-SAS-LEAS:16}
resolves the differentiability check by randomly perturbing the
gradient direction whereas~\cite{KCK:07,KCK:10} avoid it by allowing
null steps where the decision variable does not change and the
algorithm only reevaluates its parameters.  A non-derivative version
of the algorithm is further considered in~\cite{KCK:10}, which
approximates the gradient of the function via Steklov averages, 
while the recent work \cite{JVB-QL:21} extends the applicability 
of the algorithm beyond locally Lipschitz functions, requiring them only
to be directionally Lipschitz.
%

\textit{Statement of contributions:}
In this paper, we extend the GS algorithm to optimization problems
where the objective function does not have a closed-form formula and
is only approximated as the solution to an inner maximization problem
over a compact set of parameters.  Our first contribution is to
show that, almost surely,
this modified gradient sampling (mGS) algorithm guarantees asymptotic
convergence to a Clarke stationary point.  This relies on nontrivial
adjustments of derivative-free versions of the algorithm, since the
objective functions we consider are defined implicitly through the
solution of an inner maximization problem; thus, both their values
and their derivatives need to be approximated by different methods.

Our second contribution pertains to the conditions under which this
convergence is established. In particular, the recent survey paper
\cite{JVB-FEC-ASL-MLO-LS:20} noted that a necessary technical
requirement to establish convergence of the GS algorithm has been
overlooked in all existing works, which is that the objective function
needs to be \textit{continuously differentiable} on an \textit{open
  set of full measure}. Here, we significantly relax this condition by
only requiring that the objective function is of
\textit{lower-$\mathcal C^2$ class}, the class of so-called subsmooth
functions~\cite{RTR-RJBW:98}; in fact we only require the objective function to be locally Lispchitz continuous and lower-$\mathcal C^2$ on an open set of full measure. As it belongs to the lower-$\mathcal C^2$ class, the objective function may have discontinuous first-order derivatives in the classical sense on a set of positive measure. In the Appendix, we provide an explicit example of such a function.
%

Our third contribution establishes sufficient conditions that
guarantee the asymptotic convergence of the algorithm to a desirable
convex subset of the decision space towards which the negated gradient
of the function is always directed. In particular, we establish that
for any initial point, the iterations of the algorithm will get
arbitrarily close to this set. This is achieved without resorting to
any algorithm modifications that would introduce hard or soft
constraints on the feasible set, increasing its computational
complexity.

Our fourth contribution provides streamlined
probabilistic arguments to prove the almost sure convergence of the
(m)GS algorithm.  Existing arguments for this purpose establish that
certain events associated to limit points of the algorithm iterates
(cf. limit point $\bar x$ in \cite[Proof of Theorem 3.3(ii)]{KCK:07}
or constant $\eta=\inf_{k\in\mathbb N}\rho_\epsilon(x^k)$ in
\cite[Proof of Theorem 4.4]{JVB-ASL-MLO:05}) have probability
zero. Directly considering these events leads to their
parameterization through an uncountable set and it is not immediately
clear why their union will have probability zero. To elucidate that
this poses no issues, we modify the proof from \cite{KCK:07} by
building a suitable countable cover of the candidate limit points of
the algorithm. In the Appendix, we also provide a formal description
of the Markov chain generated by the algorithm to make the almost
everywhere convergence argument more precise.

Our final contribution is the exploitation of the algorithm to solve a
distributionally robust coverage optimization problem on the real
line. We establish that the objective function of this problem is
indeed lower-$\mathcal C^2$ on an open set of full measure, and
therefore meets the convergence criteria of the mGS algorithm. We also
show how a direct modification of the cost outside a convex set
ensures that the algorithm iterates will converge to it, yielding
decisions within a desirable range. An earlier version of the mGS algorithm presented here appeared without proof in the conference paper~\cite{DB-JC-SM:23-cdc}, which was focused on data-driven coverage control. The overlap with the current manuscript is minimal given the cadre of contributions outlined above, which includes the proof of the updated mGS algorithm together with several auxiliary and new results.

The paper is organized as follows. Section~\ref{sec:prelims}
introduces basic notation and preliminaries from non-smooth
analysis. Section~\ref{sec:GS:convergence} presents the mGS algorithm
and formally establishes its convergence. Section~\ref{sec:invariance}
generalizes the analysis further by providing sufficient conditions
that ensure the asymptotic convergence of the iterations to convex
domains. Section~\ref{sec:example} presents a numerical example of a
distributionally robust coverage optimization problem that is solved using the
mGS algorithm. In the Appendix, we gather the proofs of several
results that are mostly technical and omitted from the main body of
the paper to enhance the flow.

\section{Preliminaries and tools from non-smooth analysis}\label{sec:prelims}

We denote by $\|\cdot\|$ the Euclidean norm in $\Rat{n}$.  We use the
notation $[n_1:n_2]$ for the set of integers
$\{n_1,n_1+1,\ldots,n_2\}\subset\bN\cup\{0\} =:\bN_0$ and denote
$[n]:=[1:n]$. 
Given two sets $A,B\subset\Rat{n}$, we denote the smallest distance
between their elements by
$\dist(A,B):=\inf\{\|x-y\|\,|\,x\in A,y\in B\}$ and also use the
notation $\dist(x,A):=\dist(\{x\},A)$ when considering single-element
sets. The convex hull, closure, interior, and boundary of a set
$A\subset\Rat{n}$ are denoted by $\conv(A)$, $\cl(A)$, $\interior{A}$, and $\bdry(A)$, respectively. We denote by $B(x,\eps)$ the closed ball
with center $x\in\Rat{n}$ and radius $\eps>0$ and by
$\S^{n-1}(x,\eps):=\bdry(B(x,\eps))$ the $n-1$-dimensional sphere
with the same center and radius. 
The $\eps$-inflation of a set $A\subset\Rat{n}$ is
$B(A,\eps):=\cup_{x\in A}B(x,\eps)$.  
Given a nonempty set $K\subset\Rat{n}$, we denote by $\cone{K}$ the
cone generated by it, namely,
$\cone{K}:=\{\lambda x\,|\,\lambda\ge 0, x\in K\}$.  When $K$ is closed and 
convex, we denote by $\proj_K(x):=\argmin\{\|x-x'\|\,|\,x'\in K\}$ the
projection of $x\in\Rat{n}$ to $K$. A function
$f:\mathcal O\to\Rat{m}$ is $o(x)$ as $x\to 0$, where $\mathcal O$ is
an open subset of $\Rat{n}$ with $0\in\mathcal O$, if
$\lim_{x\to 0}f(x)/\|x\|$=0.  

\textit{Non-smooth analysis:} Consider a locally Lipschitz function
$f$ on $\Rat{n}$. It is known from Rademacher's theorem that $f$ is
differentiable almost everywhere. The Clarke generalized gradient of
$f$ at $x$ is defined as
$\bar\partial f(x):=\conv\{\lim_k\nabla f(x_k)\,|\,x_k\to x,x_k\in
A\}$ (following the notation
of~\cite{JVB-ASL-MLO:05,KCK:07,RTR-RJBW:98}), where $A$ is any
full-measure subset of a neighborhood of $x$ where $f$ is
differentiable. A point $x\in\Rat{n}$ is called Clarke stationary for
$f$ if $0\in\bar\partial f(x)$, which generalizes the notion of a
stationary point for continuously differentiable functions. The Goldstein $\eps$-subdifferential of $f$ at $x$ (cf. 
\cite{AAG:77}) is defined as
$\bar\partial_\eps f(x):=\conv(\bar\partial f(B(x,\eps)))$. 
Considering any set $\mathcal D_f$ of full measure on $\Rat{n}$ where
$f$ is differentiable, its Goldstein $\eps$-subdifferential at $x$ can be
approximated by the set
\begin{align}\label{G:epsilon}
  G_\eps(x):=\cl(\conv(\nabla
  f(B(x,\eps)\cap\mathcal D_f))),
\end{align}
introduced in~\cite{JVB-ASL-MLO:05}, since
$G_\eps(x)\subset\bar\partial_\eps f(x)$ and
$\bar\partial_{\eps_1} f(x)\subset G_{\eps_2}(x)$ for
$0\le\eps_1<\eps_2$. For a general function $f$ on $\Rat{n}$ (not
necessarily locally Lipschitz), the subgradient $\partial f(\bar x)$
of $f$ at $\bar x$ is defined as the set of all vectors $v\in\Rat{n}$
for which there exist sequences $x^\nu\to \bar x$ and $v_\nu\to v$ with
\begin{align*}
\liminf_{\substack{x\to x^\nu \\ x\ne x^\nu}} \frac{f(x)-f(x^\nu)+\langle v^\nu,x-x^\nu\rangle}{\|x-x^\nu\|}\ge 0,
\end{align*}
%
for all $\nu$ (cf.~\cite[Definition 8.3, Page 301]{RTR-RJBW:98}). The
following result delineates the relationship between the subgradients
and Clarke generalized gradients of locally Lipschitz functions.

\begin{prop}\longthmtitle{Subgradient/Clarke generalized gradients relationship
    \cite[\linebreak Theorem 9.61, Page
    403]{RTR-RJBW:98}}\label{prop:Clarkegradiens:sv:subgradient}
  Assume that $f$ is locally Lipschitz on the open set
  $\mathcal O\subset \Rat{n}$. Then,
  \begin{align*}
    \bar\partial f(x)=\conv(\partial f(x))  
  \end{align*}
 for all $x\in\mathcal O$. \hfill$\blacksquare$
\end{prop}

We are further interested in higher-order regularity properties of
$f$, in particular when it can be expressed as the pointwise maximum
of appropriate families of sufficiently regular functions coined as
\textit{subsmooth} functions in~\cite{RTR-RJBW:98}.

\begin{dfn}\longthmtitle{Lower-$\mathcal C^k$ functions
    \cite[Definition 10.19, Page 447]{RTR-RJBW:98}} \label{dfn:lower:Ck:fnc}
  A real-valued function $f$ on an open subset $\mathcal O$ of
  $\Rat{n}$ is lower-$\mathcal C^k$, for some $k\in\mathbb N$, if, for
  each $\bar x\in\mathcal O$, it can be represented as
  \begin{align*}
    f(x)=\max_{t\in T}f_t(x)
  \end{align*}
  for all $x$ in some neighborhood $V\equiv V(\bar x)$ of $\bar x$,
  where the functions $f_t$ are of class $\mathcal C^2$, the index set
  $T$ is a compact topological space, and each $f_t$ and its
    partial derivatives up to order $k$ are continuous in
  $(t,x)\in T\times V$.
\end{dfn}
Next, we introduce the notion of extended second-order
differentiability.

\begin{dfn}\longthmtitle{Extended second-order differentiability
    \cite[Definition 13.1, \linebreak Page
    580]{RTR-RJBW:98}}\label{dfn:extended:second:order:differentiability}
  A real-valued function $f$ on an open subset $\mathcal O$ of
  $\Rat{n}$ is twice differentiable at $\bar x\in \mathcal O$ in the
  extended sense if it is differentiable and locally Lipschitz at
  $\bar x$, and $\nabla f$ is differentiable at $\bar x$ relative to
  its domain of existence $\mathcal D_{\nabla f}$, i.e., there exists
  $A\in\Rat{n\times n}$ such that
  \begin{align*}
    \nabla f(x)=\nabla f(\bar x)+A(x-\bar x)+o(\|x-\bar x\|), \quad\textup{for 
    all}\;x\in\mathcal D_{\nabla f}. 
  \end{align*}
  %
  The unique matrix $A$ satisfying this property is called the Hessian
  of $f$ at $\bar x$ in the extended sense and is likewise denoted by
  $\nabla^2 f(\bar x)$.
\end{dfn}

Extended second-order differentiability at a point no longer requires
the function to be differentiable on a neighborhood of it. Still, this
notion enables the first-order approximation of the function's
subgradients around that point, as clarified in the following result.

\begin{thm}\longthmtitle{Extended second-order differentiability
    characterization \cite[\linebreak Theorem 13.2, Page
    581]{RTR-RJBW:98}}\label{thm:second:order:differentiability}
  A real-valued function $f$ on an open subset $\mathcal O$ of
  $\Rat{n}$ is twice differentiable at $\bar x\in \mathcal O$ in the
  extended sense iff $f$ is finite and locally lower semicontinuous at
  $\bar x$ and the subgradient mapping
  $\partial f : \mathcal O\rightrightarrows\Rat{n}$ is differentiable
  at $x$. Namely, there exist $v\in\Rat{n}$ and $A\in\Rat{n\times n}$
  such that $\partial f(\bar x)=\{v\}$ and
  \begin{align*}
    \emptyset\ne\partial f(x)\subset v+A(x-\bar x)+o(\|x-\bar x\|)B(0,1) 
  \end{align*}
  %
  %
  for all $x$ in a neighborhood of $\bar x$. In addition, necessarily
  $v=\nabla f(\bar x)$ and $A=\nabla^2 f(\bar x)$.  \hfill$\blacksquare$
\end{thm}

The next result establishes almost everywhere second-order
differentiability of lower-$\mathcal C^2$ functions in the extended
sense.

\begin{thm}\longthmtitle{Almost everywhere extended second-order
    differentiability of lo-\linebreak wer-$\mathcal C^2$ functions \cite[Theorem
    13.51, Page
    626]{RTR-RJBW:98}}\label{thm:ae:second:order:differentiability}
  Any lower-$\mathcal C^2$ function on an open set $\mathcal O$ of
  $\Rat{n}$ is twice differentiable in the extended sense at almost
  all points of $\mathcal O$, and the matrices $\nabla^2 f(x)$ are
  symmetric at all the points where they are defined.  \hfill$\blacksquare$
\end{thm}

Although this result ensures a.e.~twice differentiability in the
extended sense, it does not necessarily guarantee continuous
differentiability. This is due to the fact that the definition of
extended differentiability only considers limits over the set where
the first derivative is guaranteed to exist and this set may not be
open. We will see later that this relaxes the continuous
differentiability requirement of the original gradient sampling (GS)
algorithm that was imposed on a set of full measure, as elaborated in
the survey article~\cite[Appendix 1]{JVB-FEC-ASL-MLO-LS:20}.
%
%

According to Definition~\ref{dfn:lower:Ck:fnc}, lower-$\mathcal C^1$
functions are defined locally as the pointwise maximum of a family of
$\mathcal C^1$ functions. The following result establishes that the
subdifferential of such a function is the convex hull of the gradients
of the maximizing functions from the family.

\begin{prop}\longthmtitle{Subgradients of lower-$\mathcal C^1$ functions
    \cite[Theorem 10.31, Page
    448]{RTR-RJBW:98}}\label{prop:lowerC1:subgradient}
  Assume that $f$ is lower-$\mathcal C^1$ on the open set
  $\mathcal O\subset \Rat{n}$ and let $f(x)=\max_{t\in T}f_t(x)$
    be its local representation around some $\bar x\in\mathcal O$.
  Then
  \begin{align*}
    \partial f(x)=\conv(\{\nabla f_t(x)\,|\,t\in T(x)\}),
  \end{align*}
  where $T(x):=\argmax_{t\in T}f_t(x)$. 
\end{prop}

\section{Gradient sampling algorithm for subsmooth
  functions}\label{sec:GS:convergence}

In this section, we provide an algorithm to solve non-smooth
optimization problems where the objective function is expressed as the
pointwise maximum of a parameterized function class. The approach
builds on a modification of the GS algorithm
introduced in \cite{JVB-ASL-MLO:05} to optimize locally Lipschitz
functions.

\subsection{Objective function class--motivation from distributionally robust coverage optimization}

We assume the objective function $f:\Rat{n}\to\Rat{}$ takes the form 
\begin{align}\label{f:pointwise:max}
  f(x)=\max_{\theta\in\Theta}F(x,\theta),
\end{align}  
where $\Theta\subset\Rat{d}$ is a \textit{compact} set and
$F:\Rat{n}\times\Theta\to\Rat{}$ is continuous. 
Our motivation to consider
this class comes from distributionally robust formulations of the optimal 
coverage problem. In its typical form, this problem seeks to determine the positions $x_1,\ldots,x_J$ of $J$ agents in $\Rat{m}$ so that they minimize the spatial coverage cost 
\begin{align*}
G(x)\equiv G(x;P_\rho) & := \bE_{y\sim P_\rho}\big[\min_{j=1,\ldots,J}\|x_j-y\|^q\big] \\
&\:= \int_Q\min_{j=1,\ldots,J}\|x_j-y\|^q\rho(y)dy\equiv\int_QH(x,y)\rho(y)dy,
\end{align*}
over a bounded convex set $Q\subset\Rat{m}$, where $x=(x_1,\ldots,x_J)\in Q^J$ and  $q\ge 1$ is a penalty exponent. Here $P_\rho$ is the probability distribution of some quantity of interest over the environment $Q$, which we assume to have a density  $\rho$. In a distributionally robust version of the problem where $P_\rho$ is unknown and there is only access to $N$ i.i.d. samples $y^1,\ldots,y^N$ of it, one can build a Wasserstein ambiguity set $\mathcal P$ with the empirical distribution $P^N=\frac{1}{N}\sum_{i=1}^N\delta_{y^i}$ as a reference measure, and solve the distributionally robust problem  
\begin{align*}
\min_{x\in Q^J}\max_{P\in\mathcal P}G(x;P).
\end{align*}
In this formulation $\mathcal P$ is the Wasserstein ball $\mathcal B_p(P^N,\eps):=\{P\in \mathcal P(Q)\,|\,W_p(P^N,P)\le\eps\}$, where $\mathcal P(Q)$ denotes the set of probability distributions on $Q$ and $W_p(P^N,P)$ the Wasserstein distance of order $p\ge 1$ between $P^N$ and $P$~\cite{CV:08}. Then standard duality results from Wasserstein distributionally robust optimization (cf. \cite{PME-DK:17}, \cite{JB-KM:19}) yield the equivalent reformulation
\begin{align*}
\min_{x\in Q^J,\lambda\ge 0}\Big\{\frac{1}{N}\sum_{i=1}^N\max_{y\in Q}\{H(x,y)+\lambda(\eps-\|y^i-y\|^p)\}\Big\}.
\end{align*}
This in turn is equivalent to the problem 
\begin{align*}
\min_{x\in Q^J,\lambda\ge 0}\max_{y\equiv(y_1,\ldots,y_N)\in Q^N}\frac{1}{N}\sum_{i=1}^N(H(x,y_i)+\lambda(\eps-\|y^i-y_i\|^p))\equiv \min_{x\in Q^J,\lambda\ge 0}\max_{y\in Q^N} F((x,\lambda),y),
\end{align*}
which minimizes a function of the form \eqref{f:pointwise:max} over the convex subset $Q^J\times\RgeO$ of $\Rat{n}\equiv\Rat{mJ+1}$. This formulation entails a high-dimensional 
non-convex inner maximization problem. A robustification of the coverage problem that bypasses this complexity issue (cf. ~\cite{DB-JC-SM:23-cdc}), considers a finite-dimensional parameterization of the density through basis functions that are supported on $Q$. This yields the alternative distributionally robust coverage problem
\begin{align}\label{DRO:coverage:general}
	\min_{x\in
		Q^J}\max_{\theta\in\Theta}
	\int_Q\min_{j=1,\ldots,J}\|x_j-y\|^q\rho_\theta(y)dy.
\end{align}
Here each density $\rho_\theta$ is expressed as a linear combination 
of basis functions $\varphi_1,\ldots,\varphi_d$, which are multiplied by 
the parameters $\theta_1,\ldots,\theta_d$ that comprise the vector
$\theta\in\Theta$, where $\Theta$ is a convex subset of $\Rat{d}$. Thus, the DRO problem is equivalently written
\begin{align} \label{linear:in:theta}
	\min_{x\in Q^J}\max_{\theta\in\Theta}
	\langle c(x),\theta \rangle,\qquad c(x):=(c_1(x),\ldots,c_d(x)),  
\end{align}
which is clearly the minimization of a function of the form~\eqref{f:pointwise:max} and  entails a linear inner maximization problem. The particular basis functions that were used in \cite{DB-JC-SM:23-cdc} 
are Haar wavelets, which are piece-wise constant over rectangular partitions of the domain and facilitate the construction of data-driven 
ambiguity sets of densities (cf.~\cite{DB-JC-SM:22-cdc}).

\subsection{Modified gradient sampling algorithm}

General functions of the form \eqref{f:pointwise:max}
cannot be optimized through the classical Gradient Sampling (GS) 
algorithm since determining their value at any given point requires 
the solution of a maximization problem, which is typically only solved approximately. As a result, we cannot compute the gradients of the 
function exactly at points where it is differentiable to approximate 
its Goldstein $\eps$-subdifferential. Neither we can always check the
differentiability of $f$ at a specific point, which is required in the
original version of the GS algorithm. Further, not knowing the exact
values of the function poses challenges on how to perform line search
to determine the stepsize.
%
%

%
Our way to tackle these challenges assumes the availability of an
oracle, which represents an algorithm that carries out an approximate
maximization of $F$ with respect to $\theta$. This oracle returns, for
each $x$ and a user-defined accuracy $\delta$, a value $\theta_\star$
satisfying $\dist(\theta_\star,\theta_\star(x))<\delta$, where
\begin{align*}
  \theta_\star(x):=\argmax_{\theta\in\Theta}F(x,\theta),
\end{align*}
is \textit{the set} of maximizers of $F$ with respect to $\theta$ for
each $x\in\Rat{n}$. Since $F$ is continuous, $\theta_\star(x)$ is
nonempty. 
%

\begin{rem} 
    \longthmtitle{Oracle examples} Here we provide examples of oracles
  that are guaranteed to return a certificate in finite, and in
  several cases even polynomial, time. When $\Theta=\{1,\ldots,N\}$ is
  finite, an exact certificate $\theta_\star$ with
  ${\rm dist}(\theta_\star,\theta_\star(x))=0$ can be directly
  obtained for each $x$ by sorting the corresponding values of
  $F$. Similarly, when $F$ is linear in $\Theta$, i.e.,
  $F(x,\theta)=\sum_{i=1}^d\theta_iF_i(x)$ and $\Theta$ is a convex
  polytope, the maximization of $\theta\mapsto F(x,\theta)$ is a
  linear program. Thus, an exact solution can again be obtained using
  the simplex method.
  
  When the parameter set $\Theta$ satisfies a finite set of convex
  equalities and inequalities, and $F$ is concave in $\theta$, then a
  sequence converging to $\theta_\star(x)$ can be obtained by solving
  a finite-dimensional convex optimization problem via an iterative
  algorithm. Such algorithms include interior point methods, which
  have polynomial complexity \cite[Chapters 5]{YN:18} and augmented
  Lagrangian methods~\cite{DPB:82}. For every accuracy
  $\varepsilon>0$, such algorithms commonly return after $k$
  iterations a certificate $\theta^k$ with
  \begin{align}  \label{optimal:value:certificate}
  	F(x,\theta_{x,\star})-F(x,\theta^k)\le\varepsilon,
  \end{align}
  for some $\theta_{x,\star}\in\theta_\star(x)$. When
  $\theta\mapsto F(x,\theta)$ is further strongly concave, this point
  can have a prescribed distance from the unique optimizer $\theta_{x,\star}$ of the
  problem (in this case $\theta_\star(x)=\{\theta_{x,\star}\}$). To clarify why, note first that strong concavity of
  $F(x,\cdot)$ is equivalent to strong convexity of
  \begin{align*}
  	\theta\mapsto g(\theta):=-F(x,\theta).
  \end{align*}
  This, in turn, is equivalent to requiring that for every
  $\theta\in\Theta$ and $\zeta\in\partial g(\theta)$ (the subdifferential of $g$ at $\theta$)
  \begin{align*}
  	g(\eta)-g(\theta)\ge
  	\langle\zeta,\eta-\theta\rangle+\frac{1}{2}\rho\|\eta-\theta\|^2\quad\text{for
  		all}\;\eta\in\Theta, 
  \end{align*}
  where $\rho$ is the modulus of strong convexity (cf. \cite[Theorem
  6.1.2, Page 280]{JBHU-CL:93}). By optimality at $\theta_{x,\star}$ and
  strong concavity of $F(x,\cdot)$, we have
  \begin{align*}
  	\frac{1}{2}\rho\|\theta^k-\theta_{x,\star}\|^2
  	\le\varepsilon\iff \|\theta^k-\theta_{x,\star}\|\le
  	\Big(\frac{2\varepsilon}{\rho}\Big)^{\frac{1}{2}}, 
  \end{align*}
  which yields the desired oracle. Namely, it suffices to select
  $\varepsilon=\frac{\delta^2\rho}{2}$ for the optimal-value
  certificate~\eqref{optimal:value:certificate} to ensure that
  ${\rm dist}(\theta_\star,\theta_\star(x))\equiv{\rm dist}(\theta^k,\{\theta_{x,\star}\})\le\delta$.
  
  When $\theta\mapsto F(x,\theta)$ is only concave (as is the case for
  example in distributionally robust coverage, where it takes the
  linear in $\theta$ form \eqref{linear:in:theta}), one can instead
  consider the regularized problem
  \begin{align*}
  	\min_{x\in\Rat{n}}\max_{\theta\in\Theta} F_\rho(x,\theta),
  	\qquad F_\rho(x,\theta):=F(x,\theta)-\frac{1}{2}\rho\|\theta\|^2.
  \end{align*}
  In particular, since $\Theta$ is compact, for any $\epsilon>0$ we may select $\rho\in(0,2\epsilon/\max_{\theta\in\Theta}\|\theta\|^2]$ to get 
  \begin{align*}
  	f(x)-\epsilon\le
  	f_{\rho}(x):=\max_{\theta\in\Theta}F_{\rho}(x,\theta)\le
  	f(x)\quad\text{for all}\;x\in \Rat{n}. 
  \end{align*} 
  Then, when $f$ and $F$ satisfy the requirements of our  main convergence result (Theorem~\ref{thm:GS:convergence}), $f_\rho$ and $F_\rho$ also satisfy the same requirements, and thus the algorithm converges to a stationary point $x_\star^\rho$ of $f_\rho$, whose value is $\epsilon$-close to that of~$f$. 
\end{rem}

We also make the following assumption about the differentiability
properties of~$F$.

\begin{assumption}\longthmtitle{Regularity of
    $F$}\label{assumptions:for:F}
  
  \noindent (i) There is an open full-measure set $\mathcal D$ of
  $\Rat{n}$ on which the functions $x\mapsto F(x,\theta)$ are twice
  continuously differentiable in $x$, and $F(x,\theta)$,
  $\nabla_x F(x,\theta)$, and $\nabla_x^2 F(x,\theta)$ are 
  continuous in $(x,\theta)\in\mathcal D\times\Theta$. Hence, $f$ is
  lower-$\mathcal C^2$ on $\mathcal D$;

  \noindent (ii) For each bounded subset $S$ of $\Rat{n}$, the
  functions $x\mapsto F(x,\theta)$ are Lipschitz with respect to
  $x\in S$
  %
  %
  with constant $L_F^x(S)$, uniformly with respect to
  $\theta\in\Theta$ and, for each $x\in\Rat{n}$, the function
  $\theta\mapsto F(x,\theta)$ is Lipschitz with respect to $\theta$
  with constant $L_F^\theta(x)$;
 
  \noindent (iii) For each $x\in\mathcal D$, there exists
  $L_{\nabla_xF}^\theta(x)>0$ such that the function
  $\theta\mapsto\nabla_x F(x,\theta)$ is Lipschitz with respect to
  $\theta$ with constant $L_{\nabla_xF}^\theta(x)$.
\end{assumption}

Note that it follows from
Theorem~\ref{thm:ae:second:order:differentiability} that $f$ is twice
differentiable in the extended sense on a full-measure subset of
$\mathcal D$, and hence, also of $\Rat{n}$, which we henceforth denote
as $\mathcal D_f$. 
We also assume that $f$ is lower bounded.

\begin{assumption}\longthmtitle{Boundedness of $f$} \label{assumption:for:f}
	The function $f$ in~\eqref{f:pointwise:max} is 
	lower bounded.
\end{assumption}

We are now ready to provide our modified GS (mGS) algorithm for
functions of the form \eqref{f:pointwise:max} that satisfy
Assumptions~\ref{assumptions:for:F} and~\ref{assumption:for:f}.

\smallskip

\noindent \textbf{(Modified) Gradient Sampling Algorithm}

\smallskip
\hrule
\smallskip

\noindent \textbf{Step 0:} \longthmtitlealt{Initialization} 
\smallskip

\noindent Select $x^1\in\Rat{n}$, $\alpha,\beta,\gamma\in(0,1)$,
$\eps_1, \nu_1>0$, $\mu,\vartheta\in (0,1)$, $m\in\{n+1,n+2,\ldots\}$,
and approximation parameters $\delta_k\searrow 0$. Set
$k:=1$. 
    
  \smallskip

\noindent \textbf{Step 1:} \longthmtitlealt{Approximation of the
 Goldstein $\eps$-subdifferential by gradient sampling} \smallskip

\noindent Sample $x^{k1},\ldots,x^{km}$ independently and uniformly
from $B(x^k,\eps_k)$.  \smallskip
 
\noindent \textbf{If} $\{x^{ki}\}_{i=1}^m\not\subset\mathcal D$
\textbf{Stop}.  \smallskip 

\noindent \textbf{Else}, set 
\begin{subequations}\label{Gk:dfn}
  \begin{align}\label{Gk:dfn:first:part}
    G_k:=\conv(\{\nabla_x F(x^{k1},\theta_\star^{k1}),\ldots,
    \nabla_x F(x^{km},\theta_\star^{km})\}),
  \end{align} 
  where
  \begin{align}\label{Gk:dfn:second:part}
    \theta_\star^{ki}\in B\Big(\theta_\star(x^{ki}),
    \frac{\delta_k}{L_{\nabla_xF}^\theta(x^{ki})}\Big),\quad i=1,\ldots,m,
  \end{align}
  are determined by the oracle that maximizes $F$ with respect
  to~$\theta$. 
\end{subequations}
\smallskip
%
%

\noindent  \textbf{Step 2:} \longthmtitlealt{Search direction computation} 
\smallskip

\noindent Find the optimizer $g^k$ of the quadratic program 
\begin{align*}
  \min & \; \|g\|^2 \\
  {\rm s.t.} &\; g\in G_k.
\end{align*} 

\noindent \textbf{Step 3:} \longthmtitlealt{Sampling radius update} 
\smallskip

\noindent \textbf{If} $\|g^k\|\le\nu_k$, set $t_k:=0$,
$\nu_{k+1}:=\vartheta\nu_k$, and $\eps_{k+1}:=\mu\eps_k$, select
any $d^k\in\S^{n-1}(0,1)$, and \textbf{go to} Step 5.  \smallskip
%
%

\noindent \textbf{Else}, set $\nu_{k+1}:=\nu_k$, $\eps_{k+1}:=\eps_k$, 
and $d^k:=-g^k/\|g^k\|$.
\smallskip

\noindent \textbf{Step 4:} \longthmtitlealt{Limited Armijo line search} 
\smallskip

\noindent (i) Choose an initial step size
$t\equiv \subscr{t}{$k$,init}\ge \subscr{t}{$k$,min}:=\gamma\eps_k/3$.
\smallskip

\noindent (ii) Set the tolerance level
$c_k:=\gamma(1-\alpha)\beta\|g^k\|\eps_k/3$, and pick
\begin{subequations}\label{theta:star:bounds}
  \begin{align}
    \theta_\star^k & \in B\Big(\theta_\star(x^k),
                     \frac{c_k}{4L_F^\theta(x^k)}\Big) \\
    (\theta_\star^k)' & \in B\Big(\theta_\star(x^k+td^k),
                        \frac{c_k}{4L_F^\theta(x^k+td^k)}\Big).
  \end{align}
  \smallskip
  %
\end{subequations}
\smallskip

\noindent (iii) \textbf{If} 
\begin{align*}
  F(x^k+td^k,(\theta_\star^k)')\le F(x^k,\theta_\star^k)-\beta t\|g^k\|
  +\frac{c_k}{2},
\end{align*}
%
set $t_k:=t$ and \textbf{go to} Step 5.
\smallskip

\noindent (iv) \textbf{If} $\gamma t<\subscr{t}{$k$,min}$, set $t_k:=0$ and \textbf{go to} Step 5.
\smallskip

\noindent (v) Set $t:=\gamma t$ and \textbf{go to} (ii).  \smallskip

\noindent \textbf{Step 5:} \longthmtitlealt{Update} 
\smallskip

\noindent Set $x^{k+1}:=x^k+t_kd^k$, $k:=k+1$ and \textbf{go to} Step
1.
\smallskip
\hrule
\medskip

We next offer an intuitive description of the steps of the algorithm.
Step~0 of the mGS algorithm contains the initialization of the
decision variable and the initial tuning of the parameters. These
parameters include the tolerances $\eps_k$ for the gradient sampling
radius and $\nu_k$ for the size of the minimum-norm element of
$G_k$. The initial values of these tolerances are set in Step~0 and
their subsequent values are obtained using the discount factors $\mu$
and $\vartheta$ in Step~3 of the algorithm. Step~1 approximates the
Goldstein $\eps$-subdifferential of $f$ through the set $G_k$ generated
by approximations of sampled gradients of $F$, while Step~2 computes
the minimum-norm element of this set. Step~3 is responsible for
reducing the sampling radius and minimum-norm element tolerance when
getting closer to Clarke stationarity.  Step~4 performs a line search
to determine the gradient step using approximations of the objective
function. Finally, Step~5 updates the values of $x^k$ based on the
chosen stepsize and search direction. 

\begin{rem}\longthmtitle{Stopping criterion}
  {\rm It is worth noting that the ``if'' condition in Step~1 is
    verified over the set $\mathcal D$, where the family of functions
    $x\mapsto F(x,\theta)$ is differentiable, instead of the smaller
    set $\mathcal D_f$, where $x\mapsto f(x)$ is
    differentiable. Checking whether a sampled point belongs to
    $\mathcal D$ is in principle easier, as this set is predefined
    based on our standing assumptions about $F$ and can often be
    described by closed-form algebraic conditions. We illustrate this
    in the example of distributionally robust coverage optimization on the real line 
    in Section~\ref{sec:example}. On the other hand, although we are aware 
    of the existence of the set $\mathcal D_f$ and its properties, an explicit 
    description of it might not be available, which in practice complicates the problem of deciding whether a sampled point belongs to $\mathcal D_f$ or not.}
  \oprocend
\end{rem}

\subsection{Convergence of the mGS algorithm}
%
%

  After the parameters and initial condition of the algorithm are
  fixed at Step 0, the sampling step (Step 1) introduces randomness at
  each iteration. As a result, the runs of the algorithm generate a
  Markov chain (cf. Appendix~\ref{appendix:MP}).  The following result
  establishes the convergence properties of the mGS algorithm with respect
  to the induced probability measure of the Markov
    chain runs. 

\begin{thm}\longthmtitle{Convergence of the {mGS}
    algorithm}\label{thm:GS:convergence}
  Let $f$ of the form~\eqref{f:pointwise:max}
  satisfy Assumption~\ref{assumption:for:f}, $F$ satisfy Assumption~\ref{assumptions:for:F}, and select the parameters and initial condition of the mGS algorithm according to Step 0. Then, with probability one,
  the mGS algorithm does not stop and
  $\nu_k,\eps_k\searrow 0$.  In addition, every accumulation point of
  $\{x^k\}$ is Clarke stationary for~$f$.
\end{thm}

The proof of convergence for the mGS algorithm follows the reasoning
in~\cite{KCK:07,KCK:10}. To clarify the required technical
modifications and keep the presentation self-contained, we provide
here the proofs of the most relevant results that are needed. We start
by establishing certain further properties of the functions $f$ and
$F$ in~\eqref{f:pointwise:max} that are a consequence of
Assumption~\ref{assumptions:for:F}.

\begin{prop}\label{prop:lower:C2:implications}
\longthmtitle{Regularity implications}  
  Let $F$ satisfy Assumption~\ref{assumptions:for:F}. Then:

  \noindent (i) The gradient of $f$ is well defined for each
  $x\in \mathcal D_f$;

  \noindent (ii) The function $f$ is locally Lipschitz on $\Rat{n}$;

  \noindent (iii) For each $\bar x\in\mathcal D_f$, the gradient of
  $f$ is continuous relative to $\mathcal D_f$, i.e., for every
  $\eps>0$ there exists $\delta>0$ such that
  \begin{align*}
    \|\nabla f(\bar x)-\nabla f(x)\|<\eps \quad\textup{for all}\; x\in B(\bar 
    x,\delta)\cap \mathcal D_f.  
  \end{align*}      
\end{prop}
\begin{proof}
  Item (i) follows directly from
  Theorem~\ref{thm:second:order:differentiability}.

  \noindent For the proof of (ii), we use the fact that for each
  bounded subset $S$ of $\Rat{n}$, the functions
  $x\mapsto F(x,\theta)$ are Lipschitz with respect to $x$ with
  constant $L_F^x(S)$, uniformly with respect to
  $\theta\in\Theta$. Thus, it follows from~\cite[Proposition 9.10,
  Page 356]{RTR-RJBW:98} that $f$ is also Lipschitz on $S$ with the
  same constant $L_F^x(S)$.

  \noindent The proof of (iii) is a direct consequence of
  Definition~\ref{dfn:extended:second:order:differentiability} about
  extended second-order differentiability and the fact that $\nabla f$
  is well defined on $\mathcal D_f$ by part (i).
\end{proof}

We also leverage the following result on approximate least-norm
elements of compact convex sets.

\begin{lemma}\longthmtitle{Approximate least-norm
    elements~\cite[Lemma 3.1]{KCK:07}}\label{lemma:beta:separation}
  Let $\emptyset\ne C\subset\Rat{n}$ be compact and convex, and
  $\beta\in(0,1)$. Then, there exists $\eta>0$ such that, for each
  $u\in\Rat{n}$ with $\dist(u,C)\le\eta$ and
  $\|u\|\le \dist(0,C)+\eta$, it holds that
  $\langle v,u\rangle>\beta\|u\|^2$ for all $v\in C$. \hfill$\blacksquare$
\end{lemma}

Next we introduce some additional notation and then present the main
technical lemmas that will be used for the convergence proof of the mGS
algorithm. The proof hinges on establishing the existence of limit 
points for the iterates of the algorithm that are Clarke stationary for~$f$. 
In particular, given an arbitrary point $x_\star$, 
its proximity to $\eps$-stationarity is given by the least-norm element of
$G_\eps(x_\star)$,
\begin{align}\label{distance:rho:eps}
  \rho_\eps(x_\star):=\dist(0,G_\eps(x_\star)),
\end{align}
where $G_\eps$ is defined in~\eqref{G:epsilon}. 
The key idea behind the (m)GS algorithm
%
%
is to approximate 
this least-norm element using sampled gradients of nearby points.
To this end, we seek to characterize the set of $m$-tuples that can be 
used to approximate $\rho_\eps(x_\star)$ with a prescribed accuracy $\eta$,
leveraging only gradients of $F$.
%
We therefore define
\begin{align}\label{set:D:eps}
  D_\eps^m(x):=(\interior{B(x,\eps)}\cap\mathcal D_f)^m
  \subset \Rat{mn},
\end{align}
%
and consider the set
\begin{align}
  V_\eps(x_\star,x,\eta):= \{(y^1,\ldots,y^m)\in D_\eps^m(x)\,|\,  
  \dist(0,\conv(\{\nabla_x F
  &(y^i,\theta^i)\}_{i=1}^m)) \le 
    \;\rho_\eps(x_\star) +\eta,  \label{set:V:dfn} 
  \\
  & \textup{for all}\;\theta^i \in 
    B(\theta_\star(y^i),\psi(y^i,\eta))\}, \nonumber
\end{align}
where 
\begin{align} \label{psi:dfn}
  \psi(y,\eta):=\frac{\eta}{3L_{\nabla_xF}^\theta(y)},
  \;\textup{for any}\; y\in\mathcal D\;{\rm and}\;\eta>0. 
\end{align}
The following result characterizes the properties of~$V_\eps$. 

%
%
\begin{prop}\longthmtitle{Properties of
    $V_\eps$}\label{prop:U:properties:and:tk:bound}
  Let $\eps>0$ and $x_\star\in\Rat{n}$. Then:
  
  \noindent (i) For every $\eta>0$, there exist $\tau>0$ and a nonempty
  full-measure set $U$ satisfying $U\subset V_\eps(x_\star,x,\eta)$
  for all $x\in B(x_\star,\tau)$;
  
  \noindent (ii) Assume further that $0\notin G_\eps(x_\star)$ and
  pick $\eta>0$ such that the result of
  Lemma~\ref{lemma:beta:separation} holds with
  $C\equiv G_\eps(x_\star)$. Based on this selection, pick also
  $\tau>0$ and $U$ as in part (i). Then if at iteration $k$ of the mGS algorithm 
  Step~4 is reached
  with
  \begin{align*}
    x^k\in B(x_\star,\min\{\tau,\eps/3\}),\quad \eps_k
    \le\eps,\quad\delta_k\le\eta/3, \quad 
    {\rm and}\quad (x^{k1},\ldots,x^{km}) \in U, 
  \end{align*}
  where $x^{k1},\ldots,x^{km}$ are determined in Step 1, necessarily
  $ t_k\ge\gamma\eps_k/3$.
\end{prop}

Proposition~\ref{prop:U:properties:and:tk:bound}(i) provides
properties of the set of points close to a point $x_\star$ that can be
used to approximate the least-norm element of
$G_\eps(x_\star)$. Proposition~\ref{prop:U:properties:and:tk:bound}(ii)
is concerned with the implications of the properties of this set on
the algorithm stepsize. It implies that if the iterates of the
algorithm are close to a point $x_\star$ for which the
$\eps$-approximate gradient proxy $G_\eps(x_\star)$ is bounded away
from zero, then there is a non-negligible set to sample gradients that
enforces nontrivial lower bounds on the increments of the algorithm.
The fact that the set $U$ in Proposition~\ref{prop:U:properties:and:tk:bound} is no longer required to be open, as is the case, for instance, in~\cite[Lemma 3.2]{KCK:07}, enables us to establish convergence of the mGS algorithm without relying on the existence of an open full-measure set on which $f$ is continuously differentiable. 

Both results are used in the proof of the main theorem to establish
that certain events have probability zero. These events are related to
limit points of the algorithm iterates (cf. limit point $\bar x$ in
\cite[Proof of Theorem 3.3(ii)]{KCK:07}), while analogous
considerations are also used in \cite[Proof of Theorem
4.4]{JVB-ASL-MLO:05} where the events are captured through a function
of the algorithm iterates (cf. constant
$\eta=\inf_{k\in\mathbb N}\rho_\epsilon(x^k)$). Here we follow the
approach from \cite{KCK:07}. Still, the desirable events of
probability zero are parameterized by the corresponding limit points
of the algorithm iterates, and it is not necessarily obvious why these
limit points will never form an uncountable set. To clarify that this
poses no issues, we modify the proof strategy from \cite{KCK:07} by
selecting a countable set of regions that contain all candidate limit
points. Thus, we obtain a countable parameterization of the desired
events through these regions, which are constructed in the following
result.

\begin{corollary} \label{corollary:covering} \longthmtitle{Countable
    cover with controlled approximate subdifferential \linebreak properties}
  Fix $\eps_1,\nu_1>0$ and $\mu,\vartheta\in(0,1)$, and let
  $\beta\in(0,1)$ be the same as in the mGS algorithm. For each
  $\ell\in\mathbb N$, define $\eps(\ell):=\eps_1\mu^{\ell-1}$ and
  $\eta_a(\ell):=\nu_1\vartheta^{\ell-1}/2$.  For any $x\in\Rat{n}$
  with $0\notin G_{\eps(\ell)}(x)$, let also $\eta_b(\ell,x)$ be
  determined by Lemma~\ref{lemma:beta:separation} with
  $C\equiv G_{\eps(\ell)}(x)$ and $\beta$ as in the statement. Then
  for each $\ell\in\mathbb N$, there is a countable subset
  $\{x_{\ell,\iota}\}_{\iota\in\I_\ell}$ of $\Rat{n}$ such that:

\noindent (i) For each $\iota\in\I_\ell$: 

\noindent (ia) If $0\in G_{\eps(\ell)}(x_{\ell,\iota})$, then there
exist
$\tau(\ell,\iota)\equiv\tau(\eps(\ell),\eta_a(\ell),x_{\ell,\iota})>0$
and a full measure subset
$U(\ell,\iota)\equiv U(\eps(\ell),\eta_a(\ell),x_{\ell,\iota})$ of
$\Rat{mn}$, such that the conclusion of
Proposition~\ref{prop:U:properties:and:tk:bound}(i) is met with
$\eps\equiv\eps(\ell)$, $\eta\equiv \eta_a(\ell)$,
$x_\star\equiv x_{\ell,\iota}$, $\tau\equiv\tau(\ell,\iota)$, and
$U\equiv U(\ell,\iota)$;

\noindent (ib) If $0\notin G_{\eps(\ell)}(x_{\ell,\iota})$, then the
conclusion of Proposition~\ref{prop:U:properties:and:tk:bound}(ii) is
met with $\eps\equiv\eps(\ell)$,
$\eta\equiv\eta_b(\ell,x_{\ell,\iota})$,
$x_\star\equiv x_{\ell,\iota}$, and with
$\tau\equiv\tau(\ell,\iota)\equiv\tau(\eps(\ell),\eta_b(\ell,x_{\ell,\iota}),x_{\ell,\iota})$,
$U\equiv U(\eps(\ell),\eta_b(\ell,x_{\ell,\iota}),x_{\ell,\iota})$ as
determined by Proposition~\ref{prop:U:properties:and:tk:bound}(i);

\noindent (ii) The set
$\{B(x_{\ell,\iota},\widehat \tau(\ell,\iota))\}_{\iota\in\I_\ell}$ is
a cover of $\Rat{n}$, where each
$\widehat \tau(\ell,\iota):=\min\{\tau(\ell,\iota),\eps(\ell)/3\}/2
>0$ and $\tau(\ell,\iota)$ is given by (ia) or (ib) above, depending on
whether $0\in G_{\eps(\ell)}(x_{\ell,\iota})$ or not, respectively.
\end{corollary}
 
While the proof is a rather direct consequence of
Proposition~\ref{prop:U:properties:and:tk:bound}, we provide it as it
helps clarify how the set $U$ and tolerance $\tau$ of the proposition
depend on the point $x_\star$ and other parameters.

\begin{proof}
  Fix $\ell\in\mathbb N$ and let any $x_\star\in\Rat{n}$. Assume first
  that $0\in G_{\eps(\ell)}(x_\star)$. Since
  Proposition~\ref{prop:U:properties:and:tk:bound}(i) is valid for any
  $x_\star\in\Rat{n}$ and $\eps>0$, for $\eta_a(\ell)>0$ as given in
  the statement of the corollary, we can pick
  $\tau(\ell,x_\star)\equiv\tau(\eps(\ell),\eta_a(\ell),x_\star)>0$
  and a full measure subset
  $U(\ell,x_\star)\equiv U(\eps(\ell),\eta_a(\ell),x_\star)$ so that
  the conclusion of
  Proposition~\ref{prop:U:properties:and:tk:bound}(i) is met.  Next,
  assume that $0\notin G_{\eps(\ell)}(x_\star)$ and let
  $\eta_b(\ell,x_\star)$ as in the statement of the corollary. Then we
  can pick
  $\tau(\ell,x_\star)\equiv\tau(\eps(\ell),\eta_b(\ell,x_\star),x_\star)>0$
  and
  $U(\ell,x_\star)\equiv U(\eps(\ell),\eta_b(\ell,x_\star),x_\star)$
  so that the conclusion of
  Proposition~\ref{prop:U:properties:and:tk:bound}(ii) is met in this
  case. Based on the selection of $\tau(\ell,x_\star)$ for each case, we define $\widehat \tau(\ell,x_\star):=\min\{\tau(\ell,x_\star),\eps(\ell)/3\}/2$ for every  $x_\star\in\Rat{n}$,
  %
  which is always strictly positive.

  It follows that
  $\{\interior{B(x_\star,\widehat\tau(\ell,x_\star)}\}_{x_\star\in\Rat{n}}$
  is an open cover of $\Rat{n}$, from which we can extract a countable
  subcover, since $\Rat{n}$ is Lindel\"of.
  Taking the closure of the elements of this
  subcover establishes~(ii), while~(ia) and~(ib) follow directly from
  the above construction. In particular, we select the centers of the
  balls in the subcover to form
  $\{x_{\ell,\iota}\}_{\iota\in\I_\ell}$. This concludes the proof.
\end{proof}      
 
We also provide a result that is used to verify the convergence of the
sequence generated by the mGS algorithm. 

\begin{lemma}\longthmtitle{Convergence to a Clarke stationary
    point}\label{lemma:convergence}
  Assume that the parameters of the mGS algorithm satisfy
  $\delta_k,\eps_k\searrow 0$ and that there exists a subsequence
  $\{k_\ell\}$ with $x^{k_\ell}\to x_\star$ and $g^{k_\ell}\to
  0$. Then $0\in \bar\partial f(x_\star)$.
\end{lemma}

We are now ready to provide the proof of
Theorem~\ref{thm:GS:convergence}.

\begin{proof}[Proof of Theorem~\ref{thm:GS:convergence}]  \textit{Step~1: Boundedness of the cumulative increments.}  From
  Step~4 of the mGS algorithm, we have 
  \begin{align*}
    F(x^{k+1},(\theta_\star^k)') \le  
    F(x^k,\theta_\star^k)-\beta t_k\|g^k\|+\frac{c_k}{2},
  \end{align*}
  for some $(\theta_\star^k)'$ and $\theta_\star^k$ satisfying
  \eqref{theta:star:bounds} whenever $t_k\ne 0$.
  Thus, since by
  Assumption~\ref{assumptions:for:F}(ii) $F$ is Lipschitz with respect
  to $\theta$ at $x^k$ and $x^{k+1}$ with constants $L_F^\theta(x^k)$ and $L_F^\theta(x^{k+1})$, respectively, we get
  \begin{align*}
    f(x^{k+1})
    & \le F(x^{k+1},(\theta_\star^k)') 
      +\|f(x^{k+1})-F(x^{k+1},(\theta_\star^k)') \|
    \\
    & \le F(x^k,\theta_\star^k)-\beta 
      t_k\|g^k\|+\frac{c_k}{2}+L_F^\theta(x^{k+1})
      \frac{c_k}{4L_F^\theta(x^{k+1})}  
    \\
    & \le f(x^k)+\|F(x^k,\theta_\star^k)-f(x^k)\|-\beta 
      t_k\|g^k\|+\frac{3c_k}{4} \le f(x^k)-\beta t_k\|g^k\|+c_k.
  \end{align*}
  Here, we have also exploited the fact that
  $f(x^k)=F(x^k,\theta_\star)$ and
  $f(x^{k+1})=F(x^{k+1},\theta_\star')$ for some
  $\theta_\star\in\theta_\star(x^k)$ and
  $\theta_\star'\in\theta_\star(x^{k+1})$, in the second and last
  inequality, respectively.  Taking further into account that either $t_k=0$ and $x^{k+1}=x^k$, or that 
  $t_k\ge\gamma\eps_k/3$ and using that
  $c_k=\gamma(1-\alpha)\beta\|g^k\|\eps_k/3$, we obtain
  \begin{align}\label{f:decrease:rate}
    f(x^{k+1}) \le
    f(x^k)-\alpha\beta t_k\|g^k\|\quad\textup{for all}\;k\in\mathbb N.
  \end{align}
  Thus, since $f$ is lower bounded, summing over $k$, we get  
  \begin{align*}
    -\infty<\inf_{k\in\mathbb N}f(x^k)=\lim_kf(x^k)\le f(x^1)
    -\sum_{k=1}^\infty\alpha\beta t_k\|g^k\|
  \end{align*}
  and using also the fact that $t_k=\|x^{k+1}-x^k\|$, we obtain
  \begin{subequations} \label{summable:sequences}
    \begin{align}
      \sum_{k=1}^\infty t_k\|g^k\| & <\infty
      \\ 
      \sum_{k=1}^\infty \|x^{k+1}-x^k\|\|g^k\| & <\infty. 
    \end{align}
  \end{subequations}

  \noindent \textit{Step~2: Proof of the fact that
    $\nu_k,\eps_k\searrow 0$.}  To show that $\nu_k,\eps_k\searrow 0$
  with probability one, consider the complementary event $\mathcal E$,
  representing the algorithm runs for which
  %
  %
  there exist an integer $k_1$ and $\nu\in\{\nu_k\}$, $\eps\in\{\eps_k\}$ such that $\nu_k=\nu >0$ and $\eps_k=\eps>0$ for all $k\ge
  k_1$. 
  Then we necessarily have that $\eps=\eps_1\mu^{\ell-1}$ and
  $\nu=\nu_1\vartheta^{\ell-1}$ for some $\ell\in\mathbb N$. We denote
  each such event by $\mathcal E_{\ell,k_1}$ and fix $\ell,k_1$ from
  now on. For each algorithm run in $\mathcal E_{\ell,k_1}$, it
  follows from Step~3 of the mGS algorithm that $\|g^k\|\ge \nu$ for
  all $k\ge k_1$ and, from \eqref{summable:sequences}, we obtain that
  $t_k\to 0$ and $\sum_{k=1}^\infty\|x^{k+1}-x^k\|<\infty$, which
  implies that there exists $\widehat x$ such that
  $x^k\to \widehat x$. 
  
  By invoking Corollary~\ref{corollary:covering} for the fixed integer
  $\ell$, we can select countable sets
  $\{x_{\ell,\iota}\}_{\iota\in\I_\ell}\subset\Rat{n}$ and
  $\{\widehat\tau(\ell,\iota)\}_{\iota\in\I_\ell}\subset\RgO$, such
  that
  $\{B(x_{\ell,\iota},\widehat\tau(\ell,\iota))\}_{\iota\in\I_\ell}$
  is a countable cover of $\Rat{n}$ and the conclusion of the
  corollary is met. Thus, we have that
  $\mathcal E_{\ell,k_1}=\cup_{\iota\in\I_\ell}\mathcal
  E_{\ell,k_1,\iota}$,
  %
  %
  where each event $\mathcal E_{\ell,k_1,\iota}$ is determined
  by the additional requirement that the corresponding limit
  points of the algorithm satisfy $\widehat x\in   B(x_{\ell,\iota},\widehat\tau(\ell,\iota))$.
  %
  %
  %
  We next fix such an event
  $\mathcal E_{\ell,k_1,\iota}$ 
  and distinguish two main cases. To this end, we recall that at
  each iteration~$k$, the algorithm samples $m$ points
  $x^{k1},\ldots,x^{km}$ and determines the parameter
  $\theta_\star^{ki}$ for each point by approximately solving the
  inner maximization problem.

  \noindent \textit{Case (a): $0\in G_\eps(x_{\ell,\iota})$.} We
  denote by $\mathcal E_{\ell,k_1,\iota}^a$ the events of this
  case, namely, for which $0\in G_\eps(x_{\ell,\iota})$. 
  By assumption, the minimum-norm element $g^k$ determined in Step 2 of 
  the algorithm satisfies
  $\|g^k\|\ge\nu$, 
  which by \eqref{Gk:dfn} 
  implies that
  \begin{align}\label{convex:hull:above:nu}
    \dist(0,\conv(\{\nabla_xF(x^{ki},\theta_\star^{ki})\}_{i=1}^m))\ge\nu,
  \end{align} 
  %
  for all $k\ge k_1$, where each $\theta_\star^{ki}$ is given in
  \eqref{Gk:dfn:second:part}. We next pick
  $\eta_a(\ell) = \nu_1\vartheta^{\ell-1}/2$, according to
  Corollary~\ref{corollary:covering}. Since
  $\nu=\nu_1\vartheta^{\ell-1}$ we get
  \begin{align} \label{eta:a:value}
  \eta_a(\ell)=\nu/2.
  \end{align}
  Then it follows from Corollary~\ref{corollary:covering}(ia) and
  Proposition~\ref{prop:U:properties:and:tk:bound}(i) that there
  exists a nonzero measure subset
  $U(\ell,\iota)\equiv U(\eps,\eta_a(\ell),x_{\ell,\iota})$ of
  $\Rat{mn}$ such that
    \begin{align} \label{U:containment:for:main:thm}
    U(\ell,\iota)\subset V_\eps(x_{\ell,\iota},x,\eta_a(\ell))
    \;\textup{for all}\; x\in B(x_{\ell,\iota},2\widehat \tau(\ell,\iota)),
    \end{align}
    where $\widehat \tau(\ell,\iota)$ is determined by Corollary~\ref{corollary:covering}(ii).
    Since for each event in $\mathcal E_{\ell,k_1,\iota}^a$ there is
    $\widehat x\in B(x_{\ell,\iota},\widehat\tau(\ell,\iota))$
    with $x^k\to \widehat x$, and further $\delta_k\searrow 0$, there
    is an integer $k_2\ge k_1$ such that
    \begin{align} \label{event:characterization} x^k\in B(\widehat
      x,\widehat\tau(\ell,\iota))\;{\rm and}\;
      \delta_k\le\eta_a(\ell)/3 \;\textup{for all}\; k\ge k_2.
  \end{align}
  
  Denoting $\mathcal E_{\ell,k_1,\iota,k_2}^a$ the events
  $\mathcal E_{\ell,k_1,\iota}^a$ for which
  \eqref{event:characterization} holds for a given $k_2\ge k_1$, we
  get that
  $\mathcal E_{\ell,k_1,\iota}^a=\cup_{k_2\ge k_1}\mathcal
  E_{\ell,k_1,\iota,k_2}^a$. 
  We claim that
  \begin{align} \label{run:U:property} \textup{for each run
      from}\;\mathcal E_{\ell,k_1,\iota,k_2}^a\;\textup{necessarily}\;
    (x^{k1},\ldots,x^{km})\notin U(\ell,\iota),\;\textup{for
      all}\;k\ge k_2.
  \end{align}
  Indeed, suppose on the contrary that
  $(x^{\kappa 1},\ldots,x^{\kappa m})\in U(\ell,\iota)$ for some
  $\kappa\ge k_2$. Combining this with the fact that
  $x^\kappa\in B(x_{\ell,\iota},2\widehat\tau(\ell,\iota))$, we get
  from \eqref{U:containment:for:main:thm} that
  \begin{align} \label{xkis:inU:subset:V} (x^{\kappa
      1},\ldots,x^{\kappa m})\in U(\ell,\iota)\subset
    V_\eps(x_{\ell,\iota},x^\kappa,\eta_a(\ell)).
  \end{align} 
  Taking further into account that $\delta_\kappa\le\eta_a(\ell)/3$ by
  \eqref{event:characterization}, which together with
  \eqref{Gk:dfn:second:part} and \eqref{psi:dfn} yields
  \begin{align*}
    \dist(\theta_\star^{\kappa i},\theta_\star(x^{\kappa i}))
    \le\frac{\delta_\kappa}{L_{\nabla_x^\theta F}(x^{\kappa i})} 
    \le\frac{\eta_a(\ell)}{3L_{\nabla_x^\theta F}(x^{\kappa i})}=
    \psi(x^{\kappa i},\eta_a(\ell)),
  \end{align*}
  we get from \eqref{eta:a:value}, \eqref{xkis:inU:subset:V}, and the
  definition of $V_\eps$ in \eqref{set:V:dfn} 
  that
  \begin{align*}
    \dist(0,\conv(\{\nabla_xF(x^{\kappa i},\theta_\star^{\kappa i})\}_{i=1}^m))\le 
    \rho_\eps(x_{\ell,\iota})+\eta_a(\ell)=\nu/2,
  \end{align*}   
  where we  exploit that $\rho_\eps(x_{\ell,\iota})=0$ because
  $0\in G_\eps(x_{\ell,\iota})$.  But this contradicts
  \eqref{convex:hull:above:nu}, and we conclude that
  \eqref{run:U:property} holds. 
   
  Note next that from \eqref{set:V:dfn},
  \eqref{U:containment:for:main:thm}, \eqref{event:characterization},
  and the fact that
  $\widehat x\in B(x_{\ell,\iota},\widehat\tau(\ell,\iota))$, we get
  $U(\ell,\iota)\subset D_\eps^m(x^k)$ for all $k\ge k_2$. Thus, since
  $(x^{k1},\ldots,x^{km})$ is uniformly sampled with respect to the
  Lebesgue measure on $\cl(D_\eps^m(x^k))$ and $U(\ell,\iota)$ has
  nonzero measure in it for all $k\ge k_2$, we deduce from
  \eqref{run:U:property} that the event
  $\mathcal E_{\ell,k_1,\iota,k_2}^a$ has probability zero. We provide
  a detailed justification of this fact in
  Appendix~\ref{appendix:MP}. Since there is at most a countable
  number of indices $k_2\ge k_1$, we also deduce that the event
  $\mathcal E_{\ell,k_1,\iota}^a$ has probability zero.

  \noindent \textit{Case (b): $0\notin G_\eps(x_{\ell,\iota})$.} We
  respectively denote by $\mathcal E_{\ell,k_1,\iota}^b$ the events of
  this case. By selecting $\eta_b(\ell,x_{\ell,\iota})$ according to
  Lemma~\ref{lemma:beta:separation} with
  $C\equiv G_\eps(x_{\ell,\iota})$ and
  $\eps\equiv\eps(\ell)\equiv\eps_1\mu^{\ell-1}$, we get from
  Corollary~\ref{corollary:covering}(ib) and
  Proposition~\ref{prop:U:properties:and:tk:bound}(ii) that there is a
  nonzero measure subset $U(\ell,\iota)$ of $\Rat{mn}$ satisfying
  \eqref{U:containment:for:main:thm} with
  $\eta_b(\ell,x_{\ell,\iota})$ in place of
  $\eta_a(\ell)$. Here $\widehat \tau(\ell,\iota)$ in \eqref{U:containment:for:main:thm} is again determined by Corollary~\ref{corollary:covering}(ii).
  These results additionally guarantee that each time Step~5 of the
  mGS algorithm is reached,
  \begin{align}
    \label{four:premise:implication}
    x^k \in B(x_{\ell,\iota},2\widehat\tau(\ell,\iota)),\; \eps_k & =\eps(\equiv\eps(\ell)),\nonumber \\
    \delta_k & \le\eta_b(\ell,x_{\ell,\iota})/3,\; 
               {\rm and}\; (x^{k1},\ldots,x^{km}) \in U(\ell,\iota) \Longrightarrow\; t_k
               \ge\gamma\eps_k/3.
  \end{align}       
  Since for each event in $\mathcal E_{\ell,k_1,\iota}^b$ there is
  some $\widehat x\in B(x_{\ell,\iota},\widehat\tau(\ell,\iota))$ such
  that $x^k\to \widehat x$, and also $\delta_k\searrow 0$, and
  $t_k\to 0$, there exists $k_2\ge k_1$ so that
  \begin{align} \label{event:characterization:new} x^k\in B(\widehat
    x,\widehat\tau(\ell,\iota)),\; \eps_k=\eps,\;
    \delta_k\le\eta_b(\ell,x_{\ell,\iota})/3,\;{\rm and}\;
    t_k<\gamma\eps/3\;\textup{for all}\; k\ge k_2.
  \end{align}

  As in Case (a), we denote by $\mathcal E_{\ell,k_1,\iota,k_2}^b$ the
  set of such events for a fixed $k_2$.  We claim that for each
  algorithm run from $\mathcal E_{\ell,k_1,\iota,k_2}^b$ it holds
  $(x^{k1},\ldots,x^{km})\notin U$ for all $k\ge k_2$. Indeed,
  consider such a run and assume that otherwise
  $(x^{\kappa 1},\ldots,x^{\kappa m}) \in U(\ell,\iota)$ for
  some $\kappa\ge k_2$. Combining this with the fact that
  $\widehat x\in B(x_{\ell,\iota},\widehat\tau(\ell,\iota))$, and that
  $x^\kappa\in B(\widehat x,\widehat\tau(\ell,\iota))$,
  $\eps_\kappa=\eps$, and
  $\delta_\kappa\le\eta_b(\ell,x_{\ell,\iota})/3$ by
  \eqref{event:characterization:new}, we deduce from
  \eqref{four:premise:implication} that $t_\kappa\ge\gamma\eps/3$.
  But this contradicts the last property of
  \eqref{event:characterization:new}.
 
  For each of the events in $\mathcal E_{\ell,k_1,\iota,k_2}^b$, we
  thus have $x^k\in B(\widehat x,\widehat\tau(\ell,\iota))$ for all
  $k\ge k_2$ and some
  $\widehat x\in
  B(x_{\ell,\iota},\widehat\tau(\ell,\iota))$. Therefore, using also
  \eqref{U:containment:for:main:thm} with
  $\eta_b(\ell,x_{\ell,\iota})$ in place of $\eta_a(\ell)$, we get
  $U(\ell,\iota)\subset
  V_\eps(x_{\ell,\iota},x^k,\eta_b(\ell,x_{\ell,\iota}))$, and we obtain
  from \eqref{set:V:dfn} that $U(\ell,\iota)\subset D_\eps^m(x^k)$.
  From here, applying the same exact arguments as in Case (a), we
  deduce that the event $\mathcal E_{\ell,k_1,\iota}^b$ has also
  probability zero, which concludes Case (b).
  
  Since the index tuples $(\ell,k_1,\iota)$ range over a countable set
  and each event
  $\mathcal E_{\ell,k_1,\iota}=\mathcal
  E_{\ell,k_1,\iota}^a\cup\mathcal E_{\ell,k_1,\iota}^b$ has
  probability zero by Cases (a) and (b), the event
  \begin{align*}
  \mathcal E=\cup_{(\ell,k_1,\iota)}\mathcal E_{\ell,k_1,\iota}
  \end{align*}
  has again probability zero and we obtain that
  $\nu_k,\eps_k\searrow 0$ with probability~one.
  
  \noindent \textit{Step~3: Proof that each accumulation point is
    Clarke stationary.}  The proof of this step relies on showing that
  \begin{align*}
    \liminf_k\max\{\|x^k-x_\star\|,\|g^k\|\}=0,
  \end{align*}
  for any accumulation point $x_\star$ of $\{x^k\}$, and, then
  applying Lemma~\ref{lemma:convergence}. The arguments are exactly
  the same as those in part (iii) of the proof of
  \cite[Theorem~3.8]{KCK:10} and are thus omitted.
\end{proof}

\begin{rem}\longthmtitle{Generalized convergence guarantees for
    problems with additive structure}\label{rem:additive:cost}
  {\rm
    Assume that the objective function $F$ can be decomposed as the
    sum
    \begin{align*}
      F(x,\theta)=F_1(x,\theta)+F_2(x),
    \end{align*}
  where $F_1$ satisfies Assumption~\ref{assumptions:for:F} for some
  open set of full-measure $\D_1$ and $F_2$ is continuously
  differentiable relative to another set of full
  measure $\D_2$ (cf. Proposition~\ref{prop:lower:C2:implications}(iii) 
  for differentiability relative to a set). Assume further that in Step 1 of 
  the algorithm the points are sampled from $\D=\D_1\cap \D_2$, each gradient
  in~\eqref{Gk:dfn:first:part} has the form 
  \begin{align*}
    \nabla_x F(x^{ki},\theta_\star^{ki})
    =\nabla_x F_1(x^{ki},\theta_\star^{ki})+\nabla F_2(x^{ki}).
  \end{align*} 
  %
  %
  %
  Then, when $f$ defined as in \eqref{f:pointwise:max} is also locally
  Lipschitz and lower bounded, the
  conclusions of Theorem~\ref{thm:GS:convergence} remain
  valid. Namely, the mGS algorithm does not stop and all its
  accumulation points are Clarke stationary for~$f$.} \oprocend
  \end{rem}

\section{Attractivity guarantees for convex domains}\label{sec:invariance}

Here, we elaborate on the case where we are interested in the values 
of $f$ inside a compact subset ${\mathcal K}$ of $\Rat{n}$, and provide 
sufficient conditions that render ${\mathcal K}$ an attractive set for the 
iterates of the mGS algorithm. 
%
%
%
We assume that ${\mathcal K}=\prod_{j=1}^J{\mathcal K_j}$ for some 
compact and convex ${\mathcal K_j}\subset\Rat{n_j}$ with 
$n_1+\cdots+n_J=n$. This does not restrict
the generality of our analysis since we can always assume $J=1$. 
As a motivation for the analysis of this section consider the 
distributionally robust coverage problem \eqref{DRO:coverage:general} 
with $J$ agents on a convex workspace $Q\subset \Rat{m}$. As we are 
interested in optimizing their position exclusively within $Q$, we 
seek to apply the results of this section with $\mathcal K\equiv Q^J$. We make 
the following assumption.

\begin{assumption}\longthmtitle{Gradient cone property} \label{assumption:gradient:cone}
  Let $F$ satisfy Assumption~\ref{assumptions:for:F} and 
  assume additionally that its gradients
  satisfy the uniform bound
  \begin{align*}
    \langle \nabla_x F(x,\theta),\eta(x)\rangle \ge\alpha(x),
  \end{align*}
  for all $x\in\mathcal D\setminus {\mathcal K}$ and $\theta\in\Theta$,
  %
  %
  where
  $\eta(x):=(x-{\proj_{\mathcal K}}(x))/\|x-{\proj_{\mathcal 
  K}}(x)\|$ and $\alpha$ is 
  continuous 
  and strictly positive on $\Rat{n}\setminus {\mathcal K}$.  In addition,
  \begin{align*}
    \nabla_{x_j} F(x,\theta)\in\cone{x_j-{\mathcal K_j}},
  \end{align*}
  for all $j=1,\ldots,J$, $x=(x_1,\ldots,x_J)\in \mathcal D$, and 
  $\theta\in\Theta$.  
\end{assumption} 

We next show that, under Assumption~\ref{assumption:gradient:cone},
every sequence of points $x^k$ generated by the mGS algorithm
converges to~${\mathcal K}$. To prove this, we use the following estimate 
for the approximation of the Clarke generalized gradient through the sampled
gradients taken during Step~1 of the algorithm.

\begin{lemma}\longthmtitle{Sampled gradient direction
    bound}\label{lemma:sampled:gradient:direction:bound}
  For any $r>0$, there exists $\epsilon>0$ so that for each
  $x\in{\mathcal K_{\oplus r}^{\rm compl}}:=\Rat{n}\setminus 
  B({\mathcal K},r)$,
  %
  %
  $\{\theta^i\}_{i=1}^m\subset\Theta$, and
  $\{y^i\}_{i=1}^m\subset B(x,\epsilon)\cap \mathcal D$, where $m\in\mathbb N$, it holds
  \begin{align*}
    \langle g,\eta(x)\rangle\ge\alpha(x)/3,\;\textup{for all}\;
    g\in\conv(\{\nabla_x F(y^1,\theta^1),\ldots, \nabla_x
    F(y^m,\theta^m)\}). 
  \end{align*}
\end{lemma}

The proof is given in the Appendix. Next, we establish that an
appropriate inflation of ${\mathcal K}$ by a bounded ball that is
invariant for the algorithm under an upper bound $\tau$ on its step
size, $t_k$.  Given the lower bounds on the initial step size in Step~4
of the mGS algorithm, $\tau$ can be any positive number that is not
less than $\gamma\eps_0/3$.
%
%

\begin{prop}\longthmtitle{Invariance of an inflated
    domain}\label{prop:boundeness}
  Under Assumption~\ref{assumption:gradient:cone} and an upper bound
  $\tau\ge\gamma\eps_0/3$ on the step size $t_k$ of the mGS algorithm,
  for any $\tau^\star\ge\tau$ and
  $x^0\in\prod_{j=1}^J B({\mathcal K_j},\tau^\star+\eps)$, the
  sequence $x^k$ generated by the algorithm remains inside
  $\prod_{j=1}^JB({\mathcal K_j}, \tau^\star+\eps)$, where $\eps$ is
  equal to the maximum (initial) sampling radius $\eps_0$.
\end{prop}

The proof is also given in the Appendix. Using these auxiliary
results, we are ready to prove that the iterates of the algorithm
converge indeed to~${\mathcal K}$.

\begin{thm}\longthmtitle{Convergence of the modified GS sequence to
    ${\mathcal K}$}\label{thm:invariance}
  Under Assumption~\ref{assumption:gradient:cone}, for any initial
  point $x_0\in\Rat{n}$,
  the sequence $\{x^k\}$ of points generated by the modified
  GS algorithm converges to~${\mathcal K}$.
\end{thm}
\begin{proof}
  Let $x_0\in\Rat{n}$. Then there exists $\tau^\star\ge\tau$ such that
  $x^0\in\prod_{j=1}^JB({\mathcal K_j},\tau^\star+\eps)$, where $\tau$
  and $\eps$ can be specified as in Proposition~\ref{prop:boundeness}. By the
  same result, the iterates $x^k$ remain inside
  $\prod_{j=1}^JB({\mathcal K_j},\tau^\star+\eps)$ for all
  $k$. Thus, to prove the statement, it suffices to show that for
  every $\delta>0$, there exists an index $k_0$ so that
  $x^k\in B({\mathcal K},\delta)$ for all $k\ge k_0$. To establish
  this claim, assume on the contrary that there exist $\delta>0$ and a
  subsequence $\{x^{k_m}\}$ such that
  \begin{align*}
    x^{k_m}\in\widehat{\mathcal K}_{\oplus\delta}:= 
    \Big(\prod_{j=1}^JB({\mathcal K_j},\tau^\star+\eps)\Big) 
    \setminus\interior{B({\mathcal K},\delta)},
  \end{align*}
  for all $m$. It follows that, along this subsequence, infinitely many
  steps of the algorithm are \textit{not empty}, 
  %
  %
  which means that $x^{k_m}\ne x^{k_m+1}$ for infinitely many $m$. 
  %
  Indeed, otherwise this would imply by Step 3 of the mGS algorithm
  %
  %
  that the tolerance $\nu_{k_m}$ converges to zero. Thus, also
  $\|g^{k_m}\|\searrow 0$ and $\eps^{k_m}\searrow 0$. In addition,  
  from  Lemma~\ref{lemma:sampled:gradient:direction:bound}, there exists 
  $\epsilon>0$ such that  
  \begin{align} \label{gradients:in:callKdelta}
  x^{k_m}\in \widehat{\mathcal K}_{\oplus\delta}\;{\rm and}\;\eps^{k_m}\le\epsilon\Longrightarrow \langle g^{k_m},\eta(x^{k_m})\rangle\ge \alpha(x^{k_m})/3.
  \end{align}
  Thus, since $\eps^{k_m}\searrow 0$ and $x^{k_m}\in \widehat{\mathcal K}_{\oplus\delta}$ for all $m$, there exists $m_0\in\mathbb N$ such that 
  \begin{align}\label{g:norm:contradiction}
    \|g^{k_m}\|\ge \alpha_\star := 
    \min\{\alpha(x)/3\,|\,x\in \widehat{\mathcal K}_{\oplus\delta}\}>0,
  \end{align}
  for all $m\ge m_0$, where we used the fact that 
  $\widehat{\mathcal K}_{\oplus\delta}$ is compact in the last inequality. 
  But this contradicts the fact that $\|g^{k_m}\|\searrow 0$.
  %
  %
  We will henceforth assume without loss of generality that all steps
  $\{k_m\}$ are not empty. We distinguish two cases:

  \noindent \textit{Case (i): The increments $\|x^{k_m}-x^{k_m+1}\|$
    do not converge to zero.}
  If this is the case, then there exist $\tau_\star>0$ and a subsequence, which,
  with a slight abuse of notation we still denote by $\{x^{k_m}\}$, so that
  \begin{align*}
    \|x^{k_m}-x^{k_m+1}\|\ge\tau_\star ,
  \end{align*}
  for all $m$. We consider two further subcases:
  
  \noindent \textit{Case (ia): There exists $\nu_\star>0$ such that
    $\|g^{k_m}\|\ge\nu_\star$ for all $m$.} Then we have from
  \eqref{f:decrease:rate} in the proof of Theorem~\ref{thm:GS:convergence},
  that each iteration of the algorithm satisfies
  \begin{align*}
    f(x^{k_m})-f(x^{k_m+1})\ge \alpha\beta t_{k_m}\|g^{k_m}\|\ge \alpha\beta 
    \tau_\star\nu_\star>0.
  \end{align*}
  We therefore conclude that $f(x^{k_m})\searrow-\infty$ as $m\to\infty$, which 
  contradicts our assumption that $f$ is lower bounded.
  
  \noindent \textit{Case (ib): There exists a subsequence for which
    $\|g^{k_m}\|\to 0$ (with a slight abuse of notation).} Then, from
  Step 3 of the mGS algorithm, the sampling radius $\eps^{k_m}$
  also decreases to zero, and we deduce by \eqref{gradients:in:callKdelta}
  that \eqref{g:norm:contradiction} holds for all $m$ beyond some $m_0$,
  which contradicts the standing assumption of Case (ib).  
  
  \noindent \textit{Case (ii): The increments $\|x^{k_m}-x^{k_m+1}\|$
  converge to zero}.
  Equivalently, from Step 4 of the mGS algorithm, we have that
  $t^{k_m}\to 0$, and since $t^{k_m}\ge \gamma \eps_{k_m}/3$, also
  $\eps^{k_m}\to 0$. This implies that Step 3 is reached infinitely
  many times and, therefore, $\|g^{k_m}\|\to 0$ for a further
  subsequence (denoted similarly with a slight abuse of notation). As
  above, we deduce again that \eqref{g:norm:contradiction} holds for all $m$ beyond some $m_0$, which contradicts that $\|g^{k_m}\|\to 0$.
  %
  %
  All cases are now covered and the proof is complete.
\end{proof}

\section{Numerical example}
\label{sec:example}

In this section, we illustrate our approach on a locational
optimization problem subject to distributional uncertainty. The goal
is to optimally place a set of agents or resources over a line segment
to minimize an expected cost associated by their locations. To define
it, we first assign, for each given point in the segment, a local cost
given by the distance from the point to the agent closest to it. Then,
we evaluate the collective cost over the segment by averaging the
local costs with a probability density over the segment. This density
is used to model the probability of an event occurring on the segment
and weights the importance of locations.  The segment may represent,
for instance, a river parameterized by its length from which we seek
to collect waste that is distributed according to some probability
density function. The agents in this case represent litter collection
containers, which we optimally seek to place along the river banks to
minimize the cost of disposing the litter in them. Here the local cost
is proportional to the distance from the collected litter to the
closet container.

What we have described above is a one-dimensional instance of the
coverage control problem considered in~\cite{FB-JC-SM:09}, where a
group of agents is optimally deployed over an environment to minimize
a collective location cost. Here we consider the distributionally
robust version of the problem, where the density is uncertain, and we
seek to optimize the worst-case coverage cost over a parameterized set
of plausible densities. This has been formulated in
\cite{DB-JC-SM:23-cdc} using wavelet-based ambiguity sets that capture
a family of candidate density functions~\cite{DB-JC-SM:22-cdc}.

\subsection{Cost for an arbitrary number of agents}
\label{subsec:cost:any:agent:number}

Our general problem is to place $N$ agents on the real line to
maximize the coverage of a probability density on the segment
$I:=[0,K]$, which we assume to have integer coordinates to facilitate
notation. We further assume that the probability density is given by a
histogram that is constant on each integer interval $I_k:=[k-1,k)$ of
$I$, where it takes the respective values
$\theta_1,\ldots,\theta_K\ge 0$. The coverage cost is determined
through a penalty, which, at each point in $I$, is proportional to the
distance to the agent closest to it. For each fixed histogram density
$\rho_{\theta}:=\sum_{k=1}^K\theta_k\bone_{I_k}$ the coverage cost is
given by
\begin{align*}
  F_{\rm cover}(x,\theta):= 
  \int_0^K2\min_{n=1,\ldots,N}|x_n-y|\rho_{\theta}(y)dy, 
\end{align*}
where $x=(x_1,\ldots,x_N)\in\Rat{N}$ are the locations of the agents and
$\theta:=(\theta_1,\ldots,\theta_K)\in\Rat{K}$.
Here we assume that the probability density is uncertain and therefore 
each $\theta_k$ is only known to belong to a positive interval
$[\theta_k^-,\theta_k^+]$. This introduces a family of plausible 
densities that are parameterized by $\theta$ and yields the distributionally 
robust coverage problem
%
%
\begin{align*}
  \min_{x\in\Rat{N}}\max_{\theta\in\Theta}F_{\rm cover}(x,\theta)
  \equiv\min_{x\in\Rat{N}} f_{\rm cover}(x),  
\end{align*} 
where
\begin{align*}
  \Theta:=\big\{(\theta_1,\ldots,\theta_K)\in\Rat{K}\,|\,\theta_k^-\le 
  \theta_k\le\theta_k^+\;\textup{for each}\;k\;{\rm and}\; 
  \theta_1+\cdots+\theta_K=1\big\}. 
\end{align*} 
Namely, we optimize the coverage cost over the worst-case density 
from the parameterization.

We next establish that $F_{\rm cover}$ satisfies
Assumption~\ref{assumptions:for:F}. To this end, let
\begin{align*}
  \bar x_n:=\frac{x_n+x_{n+1}}{2}, \quad n=1,\ldots,\bar N:=N-1,
\end{align*}
%
%
assume without loss of generality that $x_1<\cdots<x_N$, and set
\begin{align*}
  \bar N_k & :=\#\{\bar x_k,k=1,\ldots,\bar N\,|\,\bar x_k\in I_k\} \\
  N_k & := \bar N_k+1.
\end{align*}  
Note that $\bar N_I:=\sum_{k=1}^K\bar N_k\le \bar N$ and each
$\bar N_k$ can also be zero. If we further denote
\begin{align*}
  I_{kn}:=[\alpha_{kn},\beta_{kn}],\quad k=1,\ldots,K,\quad n=1,\ldots,N_k,
\end{align*}  
with
\begin{align*}
  \alpha_{kn}
  & :=
    \begin{cases}
      k-1, & {\rm if}\;n=1, \\
      \bar x_{m_{kn}}, & {\rm otherwise},
    \end{cases}
  \\
  \beta_{kn}
  & := 
    \begin{cases}
      \bar x_{m_{kn}+1}, & {\rm if}\;1\le n\le N_k\;{\rm and}\;N_k\neq 1, \\
      k, &  {\rm otherwise}
    \end{cases}
\end{align*}
and   
\begin{align*}
  m_{kn} & :=m_{\min}+\sum_{i=1}^{k-1}\bar N_i+n-1,\quad
           k=1,\ldots,K-1,\quad n=1,\ldots,\bar
           N_k\qquad\Big(\sum_{i=1}^0\equiv 0\Big) \\
  m_{\min}
  &  :=\begin{cases}
          \min\{n=1,\ldots,\bar N\,|\,\bar x_n\ge 0\}, & {\rm if}\; \bar x_N\ge 0 \\
          1,  & {\rm otherwise}
       \end{cases}            
\end{align*}
it follows that 
\begin{align*}
  F_{\rm cover}(x,\theta)=\sum_{k=1}^K\sum_{n=1}^{N_k}
  \int_{\alpha_{kn}}^{\beta_{kn}}2|x_{m_{kn}+1}-y|\theta_kdy
  \equiv\sum_{k=1}^K\theta_k\sum_{n=1}^{N_k}h_{kn}(x),
\end{align*}
where   
\begin{align*}
  h_{kn}(x) 
  := & \begin{cases}
         (\beta_{kn}-x_{m_{kn}+1})^2-(\alpha_{kn}-x_{m_{kn}+1})^2, 
         & {\rm if}\;x_{m_{kn}+1}\le\alpha_{kn}, \\
         (\beta_{kn}-x_{m_{kn}+1})^2+(x_{m_{kn}+1}-\alpha_{kn})^2 , 
         & {\rm if}\;\alpha_{kn}<x_{m_{kn}+1}<\beta_{kn}, \\
         (x_{m_{kn}+1}-\alpha_{kn})^2-(x_{m_{kn}+1}-\beta_{kn})^2 , 
         & {\rm if}\;x_{m_{kn}+1}\ge\beta_{kn}.
       \end{cases}
\end{align*}  

We can obtain the exact same formula for $F_{\rm cover}$ for any other
agent configuration with distinct positions, by just permuting the
indices of the agents to match the order of their
positions. Introducing further the constraints
\begin{subequations}\label{coverage:domain:constraints}
  \begin{align}
    x_i & \ne x_j, &  i,j & \in[1:N],\quad i\ne j, \\
    x_i & \ne k, & i & \in[1:N],\quad k\in [0,K], \\ 
    x_i-x_{i+1} & \ne 2k, & i & \in[1:\bar N],\quad k\in [0,K],
  \end{align}  
\end{subequations}
and denoting as $\D$ the set of vectors $x\in\Rat{N}$ that satisfy
\eqref{coverage:domain:constraints}, it follows that this set is open
and has full measure, since each constraint in
\eqref{coverage:domain:constraints} defines a hyperplane in
$\Rat{N}$. One can then readily check that $F_{\rm cover}$ satisfies
Assumption~\ref{assumptions:for:F}, since it is determined by a smooth
closed-form expression at each point $x\in\D$. This enables us to
solve the robust coverage optimization problem using the mGS
algorithm, which is guaranteed to converge by
Theorem~\ref{thm:GS:convergence}.

\subsection{Explicit cost characterization for two agents}

Here we deal with the particular case of two agents, with positions
$x_1$ and $x_2$, for which we can obtain a closed-form expression for
$f_{\rm cover}$ and explicitly characterize its points of
discontinuity. For simplicity we consider the domain $[0,4]$ with a
histogram density equal to $\theta_1$ on $[0,2]$ and $\theta_2$ on
$[2,4]$, respectively. We then have
\begin{align}
  f_{\rm cover}(x) & =\max_{\theta\in\Theta}F_{\rm cover}(x,\theta)=
                     \max_{(\theta_1,\theta_2)\in\Theta}
                     (\theta_1 p_1(x) +\theta_2 p_2(x)) \nonumber \\ 
                   & =\max_{\theta_1\in[\max\{\theta_1^-,0.5-\theta_2^+\},\min\{\theta_1^+,0.5-\theta_2^-\}]}
                     \{(\theta_1 p_1(x)+(0.5-\theta_1)p_2(x)\} \nonumber \\
                   & =\begin{cases}
                        \max\{\theta_1^-,0.5-\theta_2^+\}p_1(x)
                        +(0.5-\max\{\theta_1^-,0.5-\theta_2^+\})p_2(x), 
                        & {\rm if}\; p_1(x)\le p_2(x) \\
                        \min\{\theta_1^+,0.5-\theta_2^-\} p_1(x)
                        +(0.5-\min\{\theta_1^+,0.5-\theta_2^-\})p_2(x),
                        & {\rm otherwise}
\end{cases} \label{explicit:cost}
\end{align} 
for all $x=(x_1,x_2)$ with $x_1\in[0,2]$ and $x_2\in[2,4]$
(cf. Figure~\ref{fig:domain:two:agent}). We provide the detailed derivation of this expression in \cite{DB-JC-SM:25-arXiv}.
\begin{figure}
  \begin{center}
  	\centering
  \subfigure[]{\includegraphics[width=.45\linewidth]{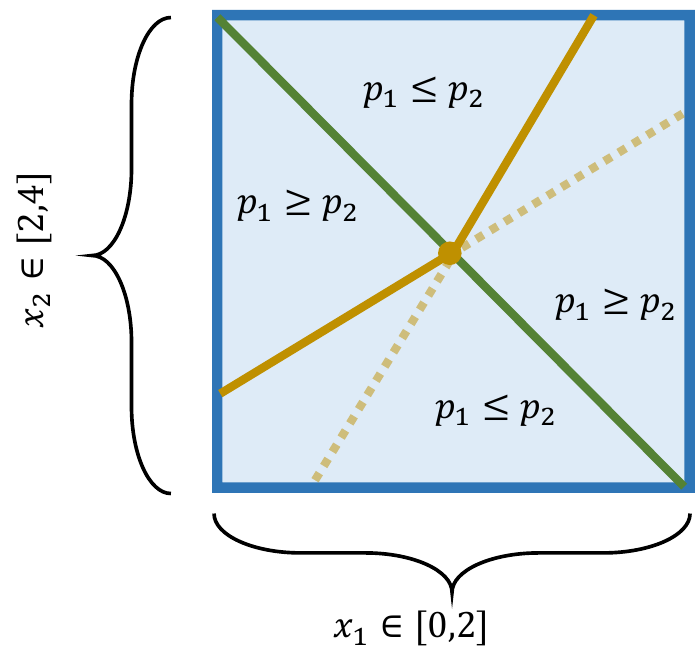}}
  \hspace{3em}
  \subfigure[]{\includegraphics[width=.45\linewidth]{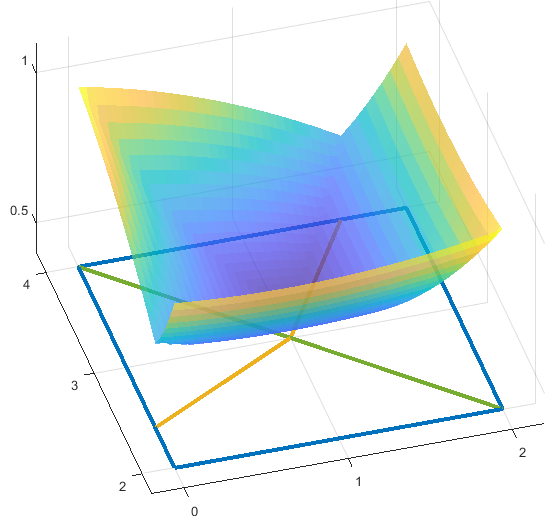}} 
  \caption{(a) The part of the two-agent domain for which we obtained
    an explicit formula of the coverage cost and the lines where this
    cost may be discontinuous, given by $p_1(x)=p_2(x)$. The green
    line $x_1+x_2-4=0$, on which $p_1(x)-p_2(x)=0$, distinguishes
    whether we are in Case (i) (below the line) or Case (ii) (above
    the line), respectively, whereas the yellow lines determine the
    additional zeros of $p_1(x)-p_2(x)$ for each respective case
    (plotted dashed in its complement). (b) shows a plot of this cost
    for the histogram value bounds $\theta_1^-=\theta_2^-=0$ and
    $\theta_1^+=\theta_2^+=0.45$.  }\label{fig:domain:two:agent}
  \end{center} 
\end{figure}
%
%

\subsection{Cost modification for attractivity}

We next modify the cost function to also ensure that the position of
the agents will converge to $\mathcal K:=\prod_{n=1}^N\mathcal K_n$,
with $\mathcal K_n:=I\equiv[0,K]$. To this end, we consider the
penalty term
\begin{align}\label{penalty:term}
  F_{\rm penalty}(x):=\sum_{n=1}^N\max\{0,-x_n,x_n-K\}
\end{align}   
and exploit Remark~\ref{rem:additive:cost}, which ensures that the mGS
algorithm still converges for the modified cost
$F:=F_{\rm cover}+F_{\rm penalty}$.  In particular, since
$F_{\rm penalty}$ is smooth on all of $\Rat{N}$ except for the
hyperplanes $x_n=0$, $x_n=K$, $n=1,\ldots,N$, it is continuously
differentiable on an open set of full measure, which establishes the
regularity requirement in Remark~\ref{rem:additive:cost}. It is also
not hard to check that
\begin{align*}
  \nabla_{x_n}F(x,\theta)= \;
  & -\bone_{\{x_n<0\}}(x)
    +\bone_{\{m_{\min}\}}(n)\bone_{\{x_n<0\}}(x)\sum_{k=1}^{\lceil \bar 
    x_{m_{\min}}\rceil}(\alpha_{k1}-\beta_{k1})\theta_k
  \\
  & + \bone_{\{m_{\max}\}}(n)\bone_{\{x_n>K\}}(x) \\
  & \quad\times\Big(\bone_{\{\bar x_{m_{\max}}\ne\lfloor\bar x_{m_{\max}}\rfloor\}}(x)
  (\beta_{\lfloor\bar x_{m_{\max}}\rfloor N_{{\lfloor\bar x_{m_{\max}}\rfloor}}}
    -\alpha_{\lfloor\bar x_{m_{\max}}\rfloor N_{{\lfloor\bar 
     x_{m_{\max}}\rfloor}}})\theta_{\lfloor\bar x_{m_{\max}}\rfloor}
  \\
  &\hspace{12em} +\sum_{k=\lceil\bar 
    x_{m_{\max}}\rceil}^{K}(\beta_{k1}-\alpha_{k1})\theta_k\Big)+\bone_{\{x_n>K\}}(x), 
\end{align*} 
where 
\begin{align*}
  m_{\max}
  &  :=\begin{cases}
         \max\{n=1,\ldots,\bar N\,|\,\bar x_n\le K\}, & {\rm if}\; \bar x_0\le K \\
         \bar N,  & {\rm otherwise},
       \end{cases}
\end{align*}
and $m_{\min}$ is introduced in Section~\ref{subsec:cost:any:agent:number}. 
Here $\bone_A$ denotes the indicator function of a set $A\subset X$,
which is one on $A$ and zero on $X\setminus A$, and is used with
$X\equiv\Rat{N}$ and $X\equiv\{1,\ldots,N\}$. To provide some intuition
behind this expression, note that when for instance $x_n<0$, the
gradient is only determined by the first line. This means that it is
simply equal to minus one, due the corresponding gradient of $-x_n$ in
$F_{\rm penalty}$, plus a positive contribution of the gradient of the
coverage cost in case the midpoint between $x_n$ and $x_{n+1}$ lies
inside the support $[0,K]$ of the distribution. It follows that
$\nabla_{x_n} F(x,\theta)<0$ when $x_n<0$ and
$\nabla_{x_n} F(x,\theta)>0$ when $x_n>k$, which implies that always
\begin{align*}
  \nabla_{x_n} F(x,\theta)\in\cone{x_n-[0,K]}=\cone{[x_n-K,x_n]}=
  \begin{cases}
    \RleO, & {\rm if}\; x_n<0, \\
    \Rat{},  & {\rm if}\; 0<x_n<K, \\
    \RgeO, & {\rm if}\; K<x_n.
  \end{cases}    
\end{align*}
Thus, Assumption~\ref{assumption:gradient:cone} is fulfilled and we
deduce from Theorem~\ref{thm:invariance} that $[0,K]^N$ is an
attractive set for the iterates of the algorithm, even when the agents
start outside it.

\subsection{Simulation results}

We validate the mGS algorithm on two simulation cases. In the first,
we optimize the two-agent coverage cost given by the explicit
formula~\eqref{explicit:cost} for agent positions with initial
condition in the square $[0,2]\times[2,4]$. This square covers all the
agent configurations on $[0,4]$, where the first agent is placed in
the first half of the interval and the second agent is placed in the
second half of it.  The sequence generated by the mGS algorithm
converges to the center of the square where the cost is minimized, as
shown in Figure~\ref{fig:two:agent:run}(a). Note that this corresponds
to the first agent reaching $1\in [0,4] $ and the second agent
reaching $3 \in [0,4]$. It is clear from the plot that, after the
sequence reaches the line of discontinuity
$\frac{5}{2} x_1-\frac{3}{2}x_2-8=0$, the descent directions selected
by the algorithm are closely aligned with it and the sequence follows
this line. On the other hand, a simple gradient descent using the same
step-size bounds and line-search parameters as the mGS algorithm fails
to converge (cf. Figure~\ref{fig:two:agent:run}(b)). This happens
because, after a certain number of iterations, the sequence reaches
the close vicinity of the line of discontinuity and cannot find a
lower function value along the negated gradient direction while
respecting the step-size limits of the algorithm.

We also optimize the coverage cost for an example with five agents and
a density that is supported on the interval $[0,6]$ with non-identical
upper and lower bounds $\theta_k^-$, $\theta_k^+$, $k=1,\ldots,6$. The
initial positions of the outmost agents are located outside this
interval. By adding the penalty term \eqref{penalty:term}, the
sequence generated by the mGS algorithm converges to optimal locations
that are strictly inside the support of the uncertain density, as
shown in Figure~\ref{fig:five:agent:run}.

\begin{figure}[htb]
  \centering
  \subfigure[]{\includegraphics[width=.35\linewidth]{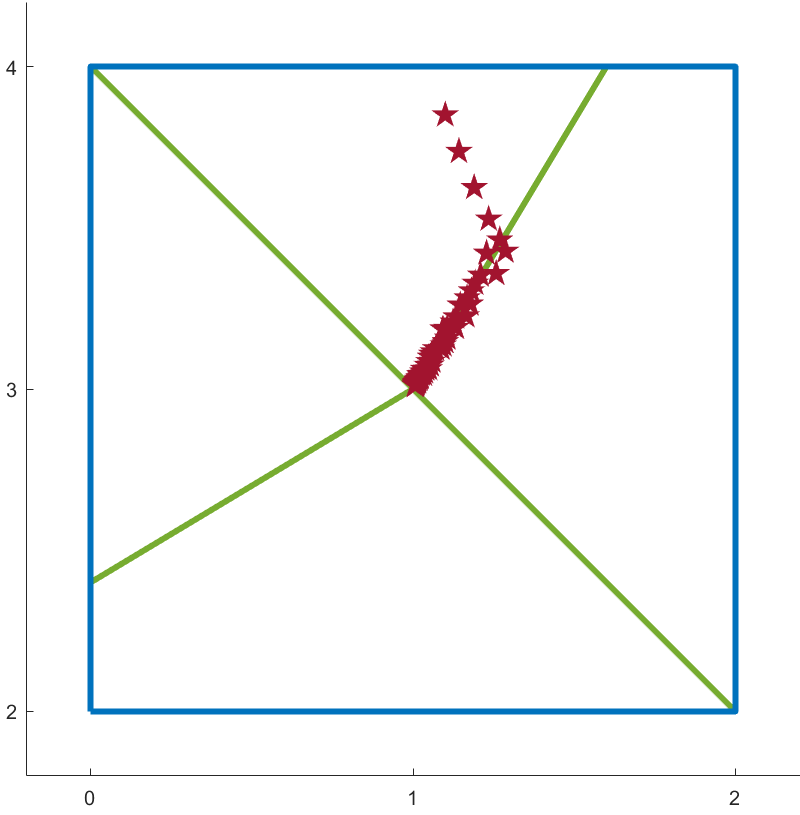}}
  \hspace{5em}
  \subfigure[]{\includegraphics[width=.35\linewidth]{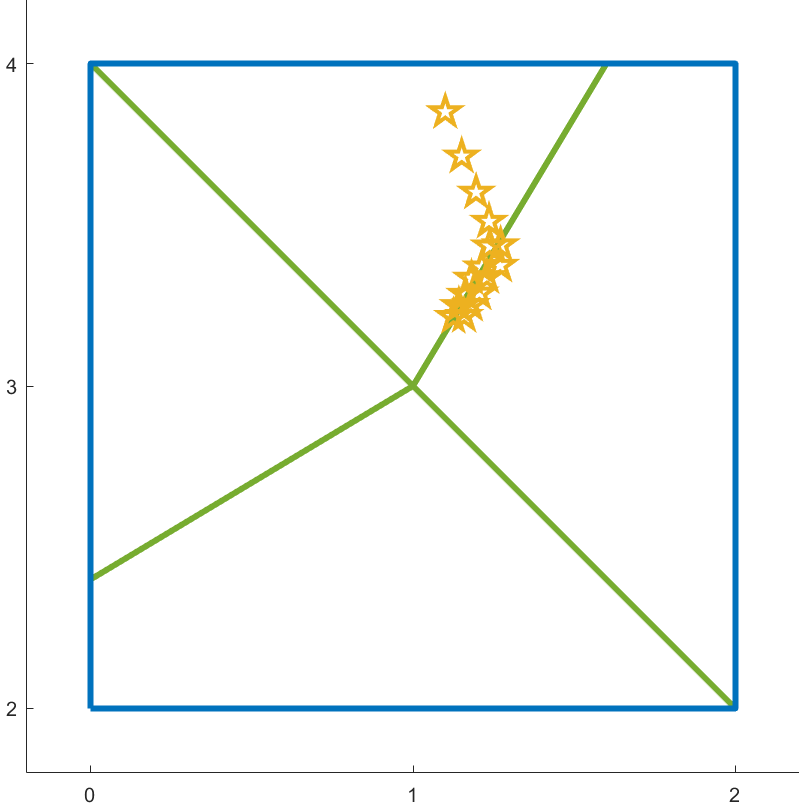}} 
  \caption{(a) shows the agent position sequence generated by the mGS
    algorithm, which converges to the minimum of the cost function (b)
    shows the same sequence generated by a simple gradient descent,
    which fails to find a descent direction and proceed further after
    a certain number of iterations.}
  \label{fig:two:agent:run}
\end{figure}

\begin{figure}[htb]
  \centering
  \includegraphics[width=.8\linewidth]{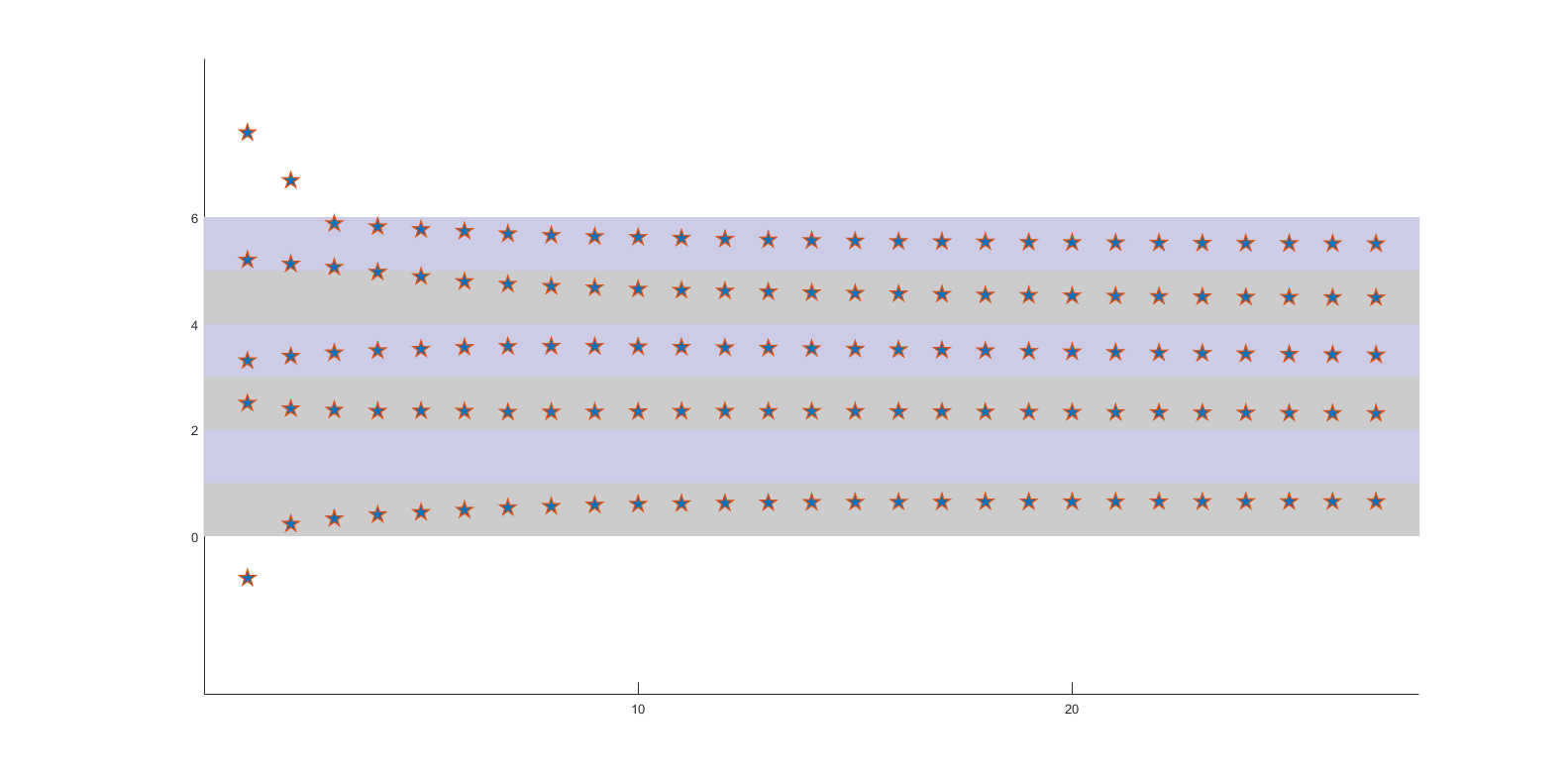}
  \caption{The plot shows the position sequence of five agents
    generated by the mGS algorithm. Here, the upper and lower bounds
    $\theta_k^-$, $\theta_k^+$, $k=1,\ldots,6$ are different. All
    agents are aligned vertically at each iteration and the two
    outmost ones lie initially outside the support $[0,6]$ of the
    uncertain density. The penalty term $F_{\rm penalty}$ ensures that
    the agents converge to an optimal worst-case configuration that
    lies inside the interval $[0,6]$.}
\label{fig:five:agent:run}
\end{figure}
%

\section{Conclusions}

Motivated by non-convex, non-smooth constrained optimization problems
where neither the values nor the gradients of the objective function
can be provided in closed form, this paper proposes a variant of the
GS algorithm which can handle objective functions that are implicitly
defined through the solution of an inner optimization problem. Our
contributions include a rigorous convergence proof of the algorithm
under assumptions that no longer require continuous differentiability
of the objective function on an open set of full measure.  We further
provide invariance guarantees for the proposed algorithm to provably
restrict its convergence over convex domains.  Future research
directions include the study of convergence for more general bilevel
programs and local sampling schemes for problems with a distributed
structure. We also aim to investigate adjustments that can help
establish finite-time convergence with prescribed probability, when
this is an acceptable alternative to almost sure convergence.

\section*{Appendix}
\renewcommand{\thesection}{A}

\subsection{Example of a lower-$\mathcal C^2$ function lacking
  continuous differentiability on a set of positive measure}

Here we provide an example to demonstrate that lower-$\mathcal C^2$
functions that are not continuously differentiable on sets of a
positive measure might arise naturally and
in fact can be obtained by taking the supremum of a nice family of
functions. To this end, we first construct a Cantor-like set by the
following recursive procedure. Let $I_0^1\equiv A_0:=[0,1]$ and
consider the midpoint $x_0^1\equiv\frac{1}{2}$ of $I_0^1$. Define
\begin{align*}
  \Delta_0^1\equiv B_0:=\Big(x_0^1-\frac{1}{2^{2+1}},x_0^1+\frac{1}{2^{2+1}}\Big)
  =\Big(\frac{1}{2}-\frac{1}{8},\frac{1}{2}+\frac{1}{8}\Big)   
\end{align*}
and the set $A_1:=A_0\setminus B_0$. Then $A_1$ consists of two
intervals $I_1^1$ and $I_1^2$ of equal length
$\frac{1}{2}(1-\frac{1}{4})\equiv\frac{1}{2}(\frac{1}{2}+\frac{1}{4})$ each. 
Next, pick the midpoints $x_1^1$
and $x_1^2$ of $I_1^1$ and $I_1^2$ and remove from them the intervals
$\Delta_1^m=\Big[x_1^m-\frac{1}{2^{4+1}},x_1^m+\frac{1}{2^{4+1}}\Big]$,
$m=1,2$, which have total length $\frac{1}{8}$. Then define the set
$B_1:=I_1^1\cup I_1^2$ and $A_2:=A_1\setminus B_1$. Proceeding
inductively, we obtain at each iteration $k$ a set of disjoint
intervals $I_k^m$, $m\in[2^k]$, of length
$\frac{1}{2^k}\Big(\frac{1}{2}+\frac{1}{2^{k+1}}\Big)$ each, which
comprise the set $A_k$. At the next iteration, we remove from each
$I_k^m$ the open interval $\Delta_k^m$, which has length
$\frac{1}{2^{2(k+1)}}$ and is centered at the midpoint $x_k^m$ of
$I_k^m$.
%
We therefore obtain a sequence of sets $A_k$ and $B_k$,
$k=1,2,\ldots$, with
\begin{align*}
  A_k & =\cup_{m=1}^{2^k} I_k^m,\hspace{2.8em} B_k =\cup_{m=1}^{2^k}
        \Delta _k^m
  \\
  A_{k+1} & =A_k\setminus B_k,\hspace{3.2em}A_0 \supset A_1\supset
            A_2\supset\cdots
  \\ 
  \lambda(A_k)
      & =\frac{1}{2}+\frac{1}{2^{k+1}},\quad
        \lambda(B_k)=\frac{1}{2^{k+2}}, 
        \quad\lambda(\Delta_k^m) =
        \frac{1}{2^{2(k+1)}}\equiv\frac{1}{2^k}\frac{1}{2^{k+2}} 
  \\
  \Delta_k^m
      & \subset
        I_k^m,\hspace{4.2em}\frac{1}{2^{k+1}}<\lambda(I_k^m)<\frac{1}{2^k}, 
\end{align*}
%
where $\lambda$ denotes the Lebesgue measure on $\Rat{}$.  It follows
by the monotonicity of the sequence $\{A_k\}$ that
$A:=\cap_{k=1}^\infty A_k$ has measure $\lambda(A)=\frac{1}{2}$. In
addition, any $x\in A$ satisfies $x\in A_k$ for all $k$. Therefore,
for every $\delta>0$, there exist indices $k$ and $m\in[2^k]$ such
that $\Delta_k^m\subset (x-\delta,x+\delta)$ by the last line of the
above properties.
%
%
Consider now a smooth function $\varphi:[-1,1]\to [-1,1]$
such that
\begin{align*}
  \supp{\varphi}\subset(-1,1),
  \quad\varphi(0)=0,\quad\varphi'(0)\ne 0.
\end{align*}
Let also $C:=\max_{x\in [-1,1]}\max\{|\varphi'(x)|,|\varphi''(x)|\}$,  
$\delta_k:=\frac{1}{2^{2(k+1)+1}}$, $\eps_k:=\frac{1}{k}\frac{\delta_k^2}{C}$, 
$k=1,2,\ldots$, and consider the real-valued smooth functions $g_k$ on $[0,1]$ 
given by  
\begin{align*}
  g_k(x):=
  \begin{cases}
    \eps_k\varphi\Big(\frac{x-x_k^m}{\delta_k}\Big),
    & {\rm
      if}\;
      x\in
      \Delta_k^m,m\in[2^k]
    \\
    0, & {\rm otherwise}
  \end{cases} 
\end{align*}
for all integers $k$. Then one can readily check that
\begin{align} \label{derivative:bounds}
  \max_{x\in[0,1]}\max\{|g_k(x)|,|g_k'(x)|,
  |g_k''(x)|\}\le\frac{1}{k}.
\end{align}
Let $f_t:[0,1]\to[0,1]$ be a family of functions that is parameterized
by $t\in[0,1]$ and associates to each $t$ between $\frac{1}{k}$ and
$\frac{1}{k+1}$ a convex combination among $g_k$ and $g_{k+1}$,
namely,
\begin{align*}
  f_t(x):=\frac{t-\frac{1}{k+1}}{k(k+1)}g_k(x)
  +\frac{\frac{1}{k}-t}{k(k+1)}g_{k+1}(x), \quad x\in[0,1],\;
  t\in\Big[\frac{1}{k+1},\frac{1}{k}\Big],\;k\in\mathbb N  
\end{align*}
We also require $f_t$
%
%
to be identically zero for $t=0$. Then $f:=\sup_{t\in[0,1]}f_t$ is
lower-$\mathcal C^2$ because its first and second derivatives with
respect to $x$ are jointly continuous in $(t,x)$. The critical part of
this assertion is establishing it for $t=0$, which follows readily
from \eqref{derivative:bounds}.  We also claim that $f$ lacks
continuous differentiability for all $x\in A$, which is a set of
positive measure. To see why this holds, it suffices to show that for
any $x\in A$ and $\delta>0$, there exists a point that is
$\delta$-close to $x$ at which the derivative of $f$ fails to
exist. Let $x\in A$ and $\delta>0$. We established above that there
exists an interval $\Delta_k^m$ such that
$\Delta_k^m\subset(x-\delta,x+\delta)$. It follows from the definition
of $f$ that $f(x)=\max\{0,g_k(x)\}$ for all $x\in\Delta_k^m$.
Therefore, since $\varphi(0)=0$ and $\varphi'(0)\ne 0$, $f$ is not
differentiable at the midpoint of $\Delta_k^m$, which concludes our
claim.

\subsection{Proofs from Section~\ref{sec:GS:convergence}}

Here, we provide the proof of
Proposition~\ref{prop:U:properties:and:tk:bound}
and~\ref{lemma:convergence} from Section~\ref{sec:GS:convergence}.
To this end, we will exploit the following properties relating the Clarke
generalized gradient, the Goldstein $\eps$-subdifferential, and the approximation $G_\eps$ of the latter given by \eqref{G:epsilon}.

\begin{lemma}\longthmtitle{Generalized gradient
    characterizations}\label{lemma:Clarke:gradient:properties}
  Let $f:\Rat{n}\to\Rat{}$ be locally Lipschitz. 
  Then the following hold:
  
  \noindent (i) The Clarke generalized gradient of $f$ at $x$ is
  equivalently given by
  \begin{align*}
    \bar\partial f(x)=\cap_{\eps>0}G_\eps(x), 
  \end{align*}
  where $G_\eps(x)$ is given by \eqref{G:epsilon};
  
  \noindent (ii) For each $\rho>0$, there exists $\eps>0$ such that
  \begin{align*}
    G_\eps(x)\subset B(\bar\partial f(x),\rho);
  \end{align*}      
  
  \noindent (iii) For any $\eps>0$, we have the inclusion 
  \begin{align*}
    G_\eps(x)\supset\cup_{x'\in\interior{B(x,\eps)}}\bar\partial f(x'); 
  \end{align*}
  %
 \noindent (iv) In fact, for any $\eps>0$, we have the stronger result 
 \begin{align*}
  G_\eps(x)=\mathring G_\eps(x):=\cl(\conv(\nabla
    f(\interior{B(x,\eps)}\cap\mathcal D_f))).
 \end{align*}
 In particular, 
  \begin{align} \label{Geps:circle:equivalent}
  \mathring G_\eps(x)=\cl(\conv(\cup_{x'\in\interior{B(x,\eps)}}\bar\partial f(x'))) 
  \end{align}
 and both $G_\eps(x)$ and $\mathring G_\eps(x)$ are independent of their defining full measure set $\D_f$ on which $f$ is differentiable.  
\end{lemma}
\begin{proof}
  Property (i) follows directly from the definitions of
  $\bar\partial f(x)$ and~$G_\eps(x)$.
  %
  %

  To prove (ii), assume by contradiction that there exists $\rho>0$
  such that, for each integer $n$, there is $u_n\in G_{1/n}(x)$ with
  $\dist(u_n,\bar\partial f(x))>\rho$. Then, by the boundedness of
  $G_\eps(x)$ a subsequence of $u_{n_k}$ converges to a point
  $v$. Since each $G_{1/n}(x)$ is closed and eventually contains the
  subsequence, $v\in G_{1/{n_k}}(x)$ for all $n_k$, and hence also
  $v\in G_\eps(x)$ for all $\eps>0$, by monotonicity of $G_\eps(x)$
  with respect to $\eps$. Thus, from part (i),
  $v\in \bar\partial f(x)$, which contradicts
  $\dist(v,\bar \partial f(x))\ge\rho$.

  To prove (iii), it suffices to show by the definition of $G_\eps(x)$
  that, for each $x'\in\interior{B(x,\eps)}$,
  $u\in\bar\partial f(x')$, and $\delta>0$, there exists
  $v\in\conv(\nabla f(B(x,\eps)\cap\mathcal D_f))$ with
  $\|u-v\|<\delta$. Consider any such triple of $x'$, $y$ and
  $\delta$, and note that by Caratheodory's theorem~\cite[Theorem
  2.29]{RTR-RJBW:98}
  %
  %
  \begin{align*}
  u = \lim_k\sum_{i=1}^{n+1}\lambda_k^iu_k^i \quad {\rm and} \quad
  u_k^i=\nabla f(x_k^i) 
  \end{align*}
  for some positive $\lambda_k^i$ with $\sum_{i=1}^{n+1}\lambda_k^i$ for each $k$. In particular, we have
  \begin{align*}
    u = \sum_{i=1}^{n+1}\lambda_iu_i \quad {\rm and} \quad
    u_i=\lim_k\nabla f(x_k^i) 
  \end{align*}
  for some positive $\lambda_i$ with $\sum_{i=1}^{n+1}\lambda_i=1$ and
  converging sequences $\{x_k^i\}\subset B(x,\eps)$. We deduce this by passing into converging subsequences for both $\{x_k^i\}$ and  $\{\lambda_k^i\}$ and making a slight abuse of notation. Next, we select an index $k$ such that $\|u_i-\nabla f(x_k^i)\|<\delta$ for all $i=1,\ldots,n+1$
  %
  %
  and set
  $v:=\sum_{i=1}^{n+1}\lambda_i\nabla f(x_k^i)$. Then we have  
  $v\in\conv(\nabla f(B(x,\eps)\cap\mathcal D_f))$ and we get that 
  \begin{align*}
    \|u-v\|\le\sum_{i=1}^n\lambda_i\|y-\nabla f(x_k^i)\|<\delta,
  \end{align*}
  which establishes the claim.   
  
  We finally prove (iv). First we show that for every $y\in\bdry(B(x,\eps))$ for which $\nabla f(y)$ is well defined, necessarily  
  $\nabla f(y)\in\mathring G_\eps(x)$. To this end, consider such a $y\in\bdry(B(x,\eps))$ and assume on the contrary that $\nabla f(y)\notin\mathring G_\eps(x)$. 
  Since $\mathring G_\eps(x)$ is compact and convex by the local Lipschitzness of $f$, we have  
  \begin{align*}
  \dist(\nabla f(y),\mathring G_\eps(x))>0,
  \end{align*}
  and we set $\delta:=\frac{1}{2}\dist(\nabla f(y),\mathring G_\eps(x))$ and
  \begin{align*}
  \eta:=\frac{\nabla f(y)-{\rm proj}_{\mathring G_\eps(x)}(\nabla f(y))}{\|\nabla f(y)-{\rm proj}_{\mathring G_\eps(x)}(\nabla f(y))\|}.
  \end{align*}
  It follows that 
  \begin{align*}
  \langle \eta,\nabla f(y)-u\rangle\ge 2\delta\;\textup{for all}\;u\in \mathring G_\eps(x),
  \end{align*}
  and, using boundedness of $\mathring G_\eps(x)$, it is not hard to show that there exists $\varphi\in(0,1)$ such that  
  \begin{align} \label{gradient:difference:along:eta:hat}
  \langle \eta,\nabla f(y)-u\rangle\ge \delta\;\textup{for all}\;u\in\mathring G_\eps(x)\;{\rm and}\;\widehat\eta\in\S^{n-1}(0,1)\;{\rm with}\;
  \langle\eta,\widehat\eta\rangle\ge\cos\varphi.
  \end{align}  
  We next select such an $\widehat\eta$, which additionally satisfies 
  $y+t\widehat\eta\in B(x,\eps)$ for all $t\in[0,s]$ and some $s>0$, or  
  $y-t\widehat\eta\in B(x,\eps)$ for all such $t$. Without loss of generality, we assume the former.
  
  By differentiability of $f$ at $y$, we can pick a sequence $t_\nu\searrow 0$ with $y+t_\nu\widehat\eta\in\interior{B(x,\eps)}$ and 
  \begin{align} \label{increment:difference:sequence}
  |f(y+t_\nu\widehat\eta)-f(y)-\langle\nabla f(y),
  t_\nu\widehat\eta\rangle|<\frac{1}{\nu}t_\nu.
  \end{align} 
  At the same time, since $f$ is locally Lipschitz, we obtain by Lebourg's mean-value theorem that 
  \begin{align*}
  f(y+t_\nu\widehat\eta)-f(y)=\langle u_\nu,t_\nu\widehat\eta\rangle
  \end{align*}
  for some $u_\nu\in\bar\partial f(z)$, where $z=\lambda y+(1-\lambda)(y+t_\nu\widehat\eta)$ and $\lambda\in(0,1)$.
  Thus, we get by \eqref{increment:difference:sequence} that
  \begin{align*} 
  |\langle\widehat\eta,\nabla f(y)-u_\nu\rangle|<\frac{1}{\nu},
  \end{align*}   
  which contradicts \eqref{gradient:difference:along:eta:hat}, since  
  $\mathring G_\eps(x)\supset G_\epsilon(x)$ for any $0<\epsilon<\eps$  by definition, 
  and therefore $u_\nu\in\mathring G_\eps(x)$ for all $\nu$ by part (iii). 
  
  For the general case, let $u\in \mathring G_\eps(x)$. Analogously to the 
  proof of part (iii), we have
  \begin{align*}
  u=\lim_k u_k\quad{\rm where}\quad u_k=\sum_{i=1}^{n+1}\lambda_k^iu_k^i
  \end{align*}
  and each $u_k^i=\nabla f(x_k^i)$ for some $x_k^i\in B(x,\eps)$. It follows from the proof so far that each $u_k^i\in G_\eps(x)$, and since $G_\eps(x)$ is closed and convex, that also $u\in G_\eps(x)$.
  
  To prove \eqref{Geps:circle:equivalent}, denote $A:=\cl(\conv(\cup_{x'\in\interior{B(x,\eps)}}\bar\partial f(x')))$ and notice that 
  $A\subset G_\eps(x)$ by part (iii). One can also directly verify that $\mathring G_\eps(x)\subset A$ by the definition of $A$.
  Thus, since  $G_\eps(x)=\mathring G_\eps(x)$ by what we proved so far in part (iv), we deduce that \eqref{Geps:circle:equivalent} holds. 
  Finally, since the definition of the Clarke generalized gradient at each point is independent of the full-measure set in its neighborhood where $f$ is differentiable (cf. \cite[Theorem 9.61, Page 403]{RTR-RJBW:98}), we also get that $\mathring G_\eps(x)$ is independent of the set $\D_f$. Since $G_\eps(x)=\mathring G_\eps(x)$, as we just proved, the same holds also for $G_\eps(x)$, which concludes the proof. 
\end{proof}

  Note that part (i) of
  Lemma~\ref{lemma:Clarke:gradient:properties} is a special case of
  Theorem~5.2 in \cite{JVB-ASL-MLO:02}, which considers functions that
  are continuous, absolutely continuous on lines, and differentiable almost
  everywhere.  Due to the more general assumptions, the
  corresponding property in \cite{JVB-ASL-MLO:02} holds only as an
  inclusion, i.e., $\bar\partial f(x)\subset\cap_{\eps>0}G_\eps(x)$.
We next specialize the discussion to the function~$f$ in
\eqref{f:pointwise:max}, which is locally Lipschitz by
Proposition~\ref{prop:lower:C2:implications}(ii).  The following
result establishes that the inflation of the Clarke generalized
gradient of $f$ contains derivatives of $F$ at nearby points when its
parameters are close to their optimal values.
%
%

\begin{lemma}\label{lemma:approximate:gradient:inclusion}
  \longthmtitle{Approximate gradient inclusion} Let
  $x_\star\in\Rat{n}$ and $\rho>0$. Then, there exists $\eps>0$ such
  that
  \begin{align*}
    \nabla_xF(x',\theta)\in B(\bar\partial f(x_\star),\rho), \quad \textup{for all}\; x'\in\interior{B(x_\star,\eps)}
    \cap\mathcal D\;{\rm and}\;
    \theta\in B(\theta_\star(x'),\psi(x',\rho)),
  \end{align*}
  %
  %
  where $\psi$ is given in \eqref{psi:dfn}.
\end{lemma}
\begin{proof}
  Note that, due to Proposition~\ref{prop:lower:C2:implications}(ii),
  $f$ is locally Lipschitz. Thus, by
  Lemma~\ref{lemma:Clarke:gradient:properties}(ii), there exists
  $\eps>0$ such that
  \begin{align}\label{G:epsilon:containment}
    G_\eps(x_\star)\subset B(\bar\partial f(x_\star),\rho/2).
  \end{align}  
  For any $x'\in\interior{B(x_\star,\eps)}\cap\mathcal D$ and
  $\theta\in B(\theta_\star(x'),\psi(x',\rho))$, there exists
  $\theta_\star\in\theta_\star(x')$ with
  $\|\theta_\star-\theta\|\le\psi(x',\rho)$. Thus, we have from
  Assumption~\ref{assumptions:for:F}(iii) that
  %
  \begin{align}\label{gradients:half:rho:distance}
    \|\nabla_x F(x',\theta)-\nabla_x F(x',\theta_\star)\|
    \le L_{\nabla_x F}^\theta(x')\|\theta-\theta_\star\|
    \le L_{\nabla_x F}^\theta(x')
    \frac{\rho}{3 L_{\nabla_x F}^\theta(x')}<\frac{\rho}{2},
  \end{align} 
  where we used the definition of $\psi$ in \eqref{psi:dfn} to obtain
  the second inequality.  Since $\theta_\star\in\theta_\star(x')$, we
  have from Proposition~\ref{prop:lowerC1:subgradient} that
  $\nabla_x F(x',\theta_\star)\in \partial f(x')$, and thus, from
  Proposition~\ref{prop:Clarkegradiens:sv:subgradient}, that also
  $\nabla_x F(x',\theta_\star)\in \bar \partial f(x')$. We therefore
  get from \eqref{gradients:half:rho:distance} that
  \begin{align*}
    \nabla_x F(x',\theta)\in B(\bar\partial f(x'),\rho/2).
  \end{align*}  
  From the latter, \eqref{G:epsilon:containment}, and
  Lemma~\ref{lemma:Clarke:gradient:properties}(iii) we deduce the
  desired result.
\end{proof}

Next, we provide the proof of
Proposition~\ref{prop:U:properties:and:tk:bound}.

\begin{proof}[Proof of Proposition~\ref{prop:U:properties:and:tk:bound}]
  \textit{Part (i).} Taking into account \eqref{distance:rho:eps}, the definition of $G_\eps(x_\star)$ in \eqref{G:epsilon}, and Lemma~\ref{lemma:Clarke:gradient:properties}(iv), we can
  select $g\in \conv(\nabla f(\interior{B(x_\star,\eps)}\cap \mathcal D_f))$ with
  \begin{align}\label{g:bound}
    \|g\|\le \rho_\eps(x_\star)+\eta/3.
  \end{align}
  Then, since $m\ge n+1$, it follows from Caratheodory's theorem and
  \eqref{set:D:eps} that there exist
  $(x_\star^1,\ldots,x_\star^m)\in D_\eps^m(x_\star)$ and
  $\lambda_1,\ldots,\lambda_m\ge 0$ with $\sum_{i=1}^m\lambda_i=1$
  such that
  \begin{align*}
    g=\sum_{i=1}^m\lambda_i\nabla f(x_\star^i).
  \end{align*}  
  By the definition of $D_\eps^m(x_\star)$, there exists
  $0< \eps'<\eps$ such that
  %
  %
  \begin{align}\label{U:prime:containment}
    U':=(\interior{B(x_\star^i, \eps')}\cap D_f)^m\subset
    D_{\eps-\eps'}(x_\star),
  \end{align} 
  and it follows that $U'$ is a full-measure set. Taking  into
  account Proposition~\ref{prop:lower:C2:implications}(i),(iii), we
  deduce that $f$ is differentiable at each $x_\star^i$ and that its
  gradient at these points is continuous relative to $\mathcal
  D_f$. Therefore, there exists $\eps''>0$ with
  \begin{align} \label{gradient:proximity}
  \|\nabla f(x_\star^i)-\nabla f(x)\|\le\frac{\eta}{3}\;\textup{for all}\;
  x\in B(x_\star^i,\eps'')\cap\mathcal D_f\;\textup{and all}\;i=1,\ldots,m.
  \end{align}
  This is the key technical point at which we relax the requirement that 
  $f$ be continuously differentiable on a full-measure set. Instead, we only require continuous differentiability relative to $\mathcal D_f$, which allows us to establish \eqref{gradient:proximity} on the full-measure set  $B(x_\star^i,\varepsilon'')\cap\mathcal D_f$.
  Next, we select $\tau:=\min\{\eps',\eps''\}$ and set
  \begin{align*}
    U:=(\interior{B(x_\star^i,\tau)}\cap\mathcal D_f)^m.
  \end{align*} 
  It follows that $U$ is a full-measure set, which establishes our
  corresponding requirement in the statement of the proposition, and
  that $U\subset D_{\eps-\tau}(x_\star)$ because of
  \eqref{U:prime:containment}.  We also obtain from \eqref{psi:dfn}, \eqref{g:bound},
  \eqref{gradient:proximity}, and
  Assumption~\ref{assumptions:for:F}(iii) that
  \begin{align}
    \Big\|\sum_{i=1}^m\lambda_i\nabla_x F(y^i,\theta^i)\Big\|
    & \le 
      \Big\|\sum_{i=1}^m\lambda_i
      \nabla
      f(x_\star^i)\Big\|
      \nonumber
    \\
    & \quad +\sum_{i=1}^m\lambda_i(\|\nabla f(x_\star^i)-\nabla 
      f(y^i)\|+\|\nabla f(y^i)-\nabla_x F(y^i,\theta^i)\|) \nonumber \\
    & \le \|g\|+\sum_{i=1}^m\lambda_i\Big(\frac{\eta}{3}
      +L_{\nabla_xF}^\theta(y^i)
      \frac{\eta}{3L_{\nabla_xF}^\theta(y^i)}\Big) \nonumber \\
    & \le \rho_\eps(x_\star)+\eta/3+2\eta/3=\rho_\eps(x_\star)+\eta,
      \label{gradient:convex:combination:bound}
  \end{align} 
  for all $(y^1,\ldots,y^m)\in U$ and
  $\theta^i \in B(\theta_\star(y^i),\psi(y^i,\eta))$. 
  %
  %
  In addition, since $U\subset D_{\eps-\tau}(x_\star)$ and for any
  $x\in B(x_\star,\tau)$ it holds that
  $D_{\eps-\tau}(x_\star)\subset D_\eps(x)$ and
  \begin{align*}
    \dist(0,\conv(\{\nabla_x F(y^i,\theta^i)\}_{i=1}^m))\le 
    \Big\|\sum_{i=1}^m\lambda_i\nabla_x F(y^i,\theta^i)\Big\|
  \end{align*}
  for all $(y^1,\ldots,y^m)\in U$ and
  $\theta^i \in B(\theta_\star(y^i),\psi(y^i,\eta))$, we deduce from
  \eqref{set:V:dfn} and \eqref{gradient:convex:combination:bound} that
  $U\subset V_\eps(x_\star,x,\eta)$.

  \noindent \textit{Part (ii).} Let $(x^{k1},\ldots,x^{km})\in
  U$ as in the statement and assume on the contrary that $t_k<\gamma\eps_k/3$. Then by the rules of Step~4, there is some $t\in[\gamma\eps_k/3\eps_k/3]$ such that the inequality of Step~4(iii) is violated. We will show that this cannot be true. To this end, since  
  $(x^{k1},\ldots,x^{km})\in V_\eps(x_\star,x_\star,\eta)$ by part (i), we get from \eqref{Gk:dfn}, \eqref{set:V:dfn}, \eqref{psi:dfn}, and the fact that $\delta_k\le \eta/3$ that
  \begin{align}\label{gk:bound}
    \|g^k\|=\dist(0,G_k)=\dist(0,\conv(\{\nabla_x
    F(x^{ki},\theta_\star^{ki})\}_{i=1}^m)) \le
    \rho_\eps(x_\star)+\eta.
  \end{align} 
  As $g^k\in G_k$, we have by \eqref{Gk:dfn:first:part} that
  \begin{align} \label{gk:decomposition}
    g^k =
    \sum_{i=1}^m\lambda_i\nabla_xF(x^{ki},\theta_\star^{ki})  
    =
    \sum_{i=1}^m\lambda_i\nabla f(x^{ki})        
    +
    \sum_{i=1}^m\lambda_i(\nabla_xF(x^{ki},\theta_\star^{ki})-\nabla f(x^{ki}))
  \end{align} 
  for some $\lambda_1,\ldots,\lambda_m\ge 0$ with
  $\sum_{i=1}^m\lambda_i=1$. Taking also into account that
  \begin{align*}
    \dist(\theta_\star^{ki},\theta_\star(x^{ki}))\le
    \frac{\delta_k}{ L_{\nabla_xF}^\theta(x^{ki})},
  \end{align*}
  we obtain from Assumption~\ref{assumptions:for:F}(iii) that
  \begin{align}  \label{gk:decomposition:bound:second:term}
    \Big\|\sum_{i=1}^m\lambda_i(\nabla_xF(x^{ki},\theta_\star^{ki})-\nabla 
    f(x^{ki}))\Big\|
 \le \sum_{i=1}^m\lambda_i\frac{\delta_k}{L_{\nabla_xF}^\theta(x^{ki})} 
      L_{\nabla_xF}^\theta(x^{ki})\le\frac{\eta}{3}\le\eta.
  \end{align}   
  Further, we have from \eqref{set:V:dfn}
    that $V_\eps(x_\star,x_\star,\eta) \subset D_\eps(x_\star)$.
    Thus, $(x^{k1},\ldots,x^{km})\in D_\eps(x_\star)$, and we obtain from
    \eqref{G:epsilon} and \eqref{set:D:eps} that
    $\sum_{i=1}^m\lambda_i\nabla f(x^{ki})\in G_\eps(x_\star)$.
  Consequently, we get from \eqref{gk:decomposition} and \eqref{gk:decomposition:bound:second:term} that $\dist(g^k,G_\eps(x_\star))\le\eta$. Combining this with \eqref{distance:rho:eps}, \eqref{gk:bound}, and our assumption that Lemma~\ref{lemma:beta:separation} holds with $C\equiv G_\eps(x_\star)$, it follows that 
  \begin{align}\label{separation:in:Gepsilon} 
    \langle v,g^k\rangle>\beta\|g^k\|^2,\quad\;\textup{for all}\; v\in G_\eps(x_\star).  
  \end{align} 
  Next, for any 
  $t\in[\gamma\eps_k/3,\eps_k/3]$, we obtain from Lebourg's mean
  value theorem~\cite[Theorem 10.17, Page 201]{FC:13} that
  \begin{align}\label{mean:value:thm}
    f(x^k+td^k)-f(x^k)=t\langle v,d^k\rangle,
  \end{align}
  where $v\in\bar\partial f(x)$ and
  $x=\lambda (x^k+td^k)+(1-\lambda)x^k$ for certain
  $\lambda\in[0,1]$. Since $d^k= -g^k/\|g^k\|$
  %
  %
  and $\eps_k\le\eps$ imply that
  $t\|d^k\|\le\eps/3$, and also $\|x^k-x_\star\|\le \eps/3$ by assumption,
  we obtain that $x\in B(x_\star,2\eps/3)$. Hence, $v\in G_\eps(x_\star)$ by
  \eqref{G:epsilon} and Lemma~\ref{lemma:Clarke:gradient:properties}(iii), and it follows from \eqref{separation:in:Gepsilon}
  %
  that $\langle v,d^k\rangle<-\beta \|g^k\|$.
    %
  %
  Thus, 
  \eqref{mean:value:thm} yields
  \begin{align*}
    f(x^k+td^k)<f(x^k)-\beta t \|g^k\|
  \end{align*}
  and we get from Assumption~\ref{assumptions:for:F}(ii) that 
  \begin{align}
    F(x^k+td^k,(\theta_\star^k)')
    & \le 
      f(x^k+td^k) +|F(x^k+td^k,(\theta_\star^k)')-f(x^k+td^k)| \nonumber
    \\
    & \le 
      f(x^k)-\beta t \|g^k\|+|F(x^k+td^k,(\theta_\star^k)')-f(x^k+td^k)|
      \nonumber \\
    & \le F(x^k,\theta_\star^k)-\beta    
      t \|g^k\|+|F(x^k+td^k,(\theta_\star^k)')-f(x^k+td^k)| \nonumber \\ 
    & \quad+|f(x^k)-F(x^k,\theta_\star^k)| \nonumber \\
    & \le F(x^k,\theta_\star^k)-\beta t\|g^k\| 
      +L_F^\theta(x^k+td^k)\dist((\theta_\star^k)',\theta_\star(x^k+td^k)) 
      \nonumber \\
    & \quad+L_F^\theta(x^k)\dist(\theta_\star^k,\theta_\star(x^k)) \le 
      F(x^k,\theta_\star^k)-\beta t\|g^k\|+\frac{c_k}{2} 
      \label{approximate:decrease:condition:derivation}          
  \end{align}
  for any $\theta_\star^k$ and $(\theta_\star^k)'$ satisfying
  \eqref{theta:star:bounds}.  It follows that the inequality of Step~4(iii) is valid for all $t\in [\gamma\eps_k/3,\eps_k/3]$, which provides the desired contradiction. We conclude that $t_k\ge\gamma\eps_k/3$.
\end{proof}
 
\begin{proof}[Proof of Lemma~\ref{lemma:convergence}]
  Assume on the contrary that
  $\dist(0,\bar\partial f(x_\star))=\rho>0$.  From
  \linebreak Lemma~\ref{lemma:approximate:gradient:inclusion}, there exists
  $\eps>0$ such that
  \begin{align}\label{approximate:gradient:inclusion}
    \nabla_xF(x',\theta)\in 
    B(\bar\partial f(x_\star),\rho/2) \quad \text{for all }
    x'\in\interior{B(x_\star,\eps)}\cap\mathcal D,  
    \theta\in B(\theta_\star(x'),\psi(x',\rho/2)).
  \end{align}
  Since $\eps_k\searrow 0$ and $x_{k_\ell}\to x_\star$,
  $\delta_k\searrow 0$, and $g_{k_\ell}\to 0$, there exists an index
  $\ell$ such that $B(x^{k_\ell},\eps_{k_\ell})\subset\interior{B(x_\star,\eps)}$,
  $\delta_{k_\ell}\le\rho/6$, and $\|g^{k_\ell}\|<\rho/2$.  Taking
  into account that the sampled points in Step~1 of the algorithm satisfy $x^{k_\ell i}\in B(x^{k_\ell},\eps_{k_\ell})$,
  $i=1,\ldots,m$, we obtain from
  $B(x^{k_\ell},\eps_{k_\ell})\subset \interior{B(x_\star,\eps)}$,
  $\delta_{k_\ell}\le\rho/6$, \eqref{Gk:dfn:second:part}, \eqref{psi:dfn}, and
  \eqref{approximate:gradient:inclusion} that
  \begin{align*}
    \nabla_xF(x^{k_\ell i},\theta_\star^{k_\ell i})\in B(\bar\partial 
    f(x_\star),\rho/2),
  \end{align*}
  for all $i=1,\ldots,m$. Thus, by convexity of $B(\bar\partial 
      f(x_\star),\rho/2)$ and monotonicity of the convex hull operation with respect to set inclusion, we deduce from \eqref{Gk:dfn:first:part} that
  \begin{align*}
    G_{k_\ell}\subset B(\bar\partial f(x_\star),\rho/2).
  \end{align*}
  Then it follows from the definition of $g^k$ in Step~2 of the algorithm that also 
  \begin{align*}
    g^{k_\ell}\in B(\bar\partial f(x_\star),\rho/2). 
  \end{align*}
  Combining this with the fact that $\|g^{k_\ell}\|<\rho/2$, we reach
  the contradiction that $\dist(0,\bar\partial f(x_\star))<\rho$, which concludes the proof. 
\end{proof}

\subsection{Proofs from Section~\ref{sec:invariance}}

Here we provide the proof of
Lemma~\ref{lemma:sampled:gradient:direction:bound} and
Proposition~\ref{prop:boundeness}.  To do so, we use the following
bound on the deviation between nearby direction vectors that
point to the closest element of a convex set.

\begin{lemma}\longthmtitle{Maximum normal deviation around
    ball}\label{lemma:normal:deviation:around:ball}
  Let $r>0$ and $\epsilon\in(0,r)$. Then, for all
  $x\in {\mathcal K_{\oplus r}^{\rm compl}}$, where
  $\mathcal K_{\oplus r}^{\rm compl}$ is given in
  Lemma~\ref{lemma:sampled:gradient:direction:bound},
  %
  %
  we have
  \begin{align} \label{eta:increment:bound}
    \|\eta(x)-\eta(y)\|\le \sqrt{2}\sqrt{1-\sqrt{1-(\epsilon/r)^2}}
    \quad\textup{for all}\;y\in B(x,\epsilon).
  \end{align}
\end{lemma}
\begin{proof}
  Let $x\in {\mathcal K_{\oplus r}^{\rm compl}}$. Then $\|x-x'\|>r$, where 
  $x':=\proj_{\mathcal K}(x)$.
  %
  %
  %
  %
  Using a scaling argument and the fact that we can accordingly
  enlarge $B(x,\epsilon)$ when $x$ is strictly farther than $r$ from
  ${\mathcal K}$, we may assume without loss of generality that 
  $\|x-x'\|=1$, $r=1$, and show the result with $\epsilon/r\equiv\epsilon$.  
  Next, let $y\in B(x,\epsilon)$ and denote $y':=\proj_{\mathcal K}(y)$. 
  We claim that there exists $z\in B(x,\epsilon)$ such that $z-x'=y-y'$ and
  therefore
  \begin{align*}
    \eta(y)=\frac{y-y'}{\|y-y'\|}=\frac{z-x'}{\|z-x'\|}.
  \end{align*} 
  Indeed, denoting $\zeta=x-x'$, $\chi=y'-x'$, $\xi=y-y'$, and taking 
  into account that $\zeta$ and $\xi$ are the outward pointing normal vectors
  to the supporting hyperplanes of $\mathcal K$ at $x'$ and $y'$, 
  respectively, we deduce that
  \begin{align} \label{normal:vetcor:angles}
  \langle\chi,\zeta\rangle\le 0 \quad {\rm and}\quad \langle\chi,\xi\rangle\ge 0.
  \end{align} 
  In addition, as $\|y-x\|\le r$ and $y-x=\xi+\chi-\zeta$, we get that 
  \begin{align*}
  \|\xi+\chi-\zeta\|^2\le \epsilon^2 & \Longrightarrow \|\xi-\zeta\|^2+\|\chi\|^2
  +2\langle\chi,\xi-\zeta\rangle\le \epsilon^2 \\
  & \Longrightarrow \|\xi-\zeta\|^2+\|\chi\|^2
  +2\langle\chi,\xi\rangle-2\langle\chi,\zeta\rangle\le \epsilon^2 
  \Longrightarrow \|\xi-\zeta\|^2\le \epsilon^2, 
  \end{align*}
  where we exploited that all the omitted terms in the last implication
  are nonnegative by \eqref{normal:vetcor:angles}. Since $\xi-\zeta=z-x$, 
  we obtain that $z\in B(x,\epsilon)$. 
      

  As a result, we can bound the distance between $\eta(x)$ and
  $\eta(y)$ by the maximum distance between $\eta(x)$ and all unit
  vectors $(z-x')/\|z-x'\|$ with $z\in B(x,\epsilon)$. This distance
  is actually attained for some vector $z_\star$ on the boundary of
  this ball. Consider now the two-dimensional plane that passes
  through $x$, $x'$, and $z_\star$, and a Euclidean coordinate system
  $(u,v)$ on this plane such that $x'=(0,0)$, $x=(0,1)$, and
  $z_\star=(u_\star,v_\star)$. Then $z_\star$ is determined by the
  solution to the optimization problem
  \begin{align*}
    \max_{(u,v)}\Big\|\frac{(u,v)}{\|(u,v)\|}-(0,1)\Big\|^2,
  \end{align*}        
  subject to the constraint that $(u,v)$ belongs the intersection of
  the plane $H$ and the boundary of the ball $B(x,\epsilon)$, namely,
  that $u^2+(v-1)^2=\epsilon^2$. Rewriting the function that we seek
  to maximize as
  \begin{align*}
    h(u,v) =\frac{u^2+(v-\|(u,v)\|)^2}{\|(u,v)\|^2} 
    =\frac{u^2+v^2-2v\sqrt{u^2+v^2}+u^2+v^2}{u^2+v^2}=2-\frac{2v}{\sqrt{u^2+v^2}}
  \end{align*} 
  and taking into account the constraint, it suffices to maximize
  $g(v):=2-2v/\sqrt{\epsilon^2+2v-1}$. Computing the sign of its
  derivative, we have
  \begin{align*}
    \sign g'(v)=\sign\Big(-\sqrt{\epsilon^2+2v-1} 
    +\frac{v2}{2\sqrt{\epsilon^2+2v-1}}\Big)=\sign(-\epsilon^2-2v+1+v). 
  \end{align*}
  Thus, we attain the maximum for
  $v_\star=1-\epsilon^2$, which lies in inside the feasible set 
  $[1-\epsilon,1+\epsilon]$ of $v$. It follows that  
  $u_\star=\sqrt{\epsilon^2-\epsilon^4}$ and we get
  \begin{align*}
    \frac{u_\star^2+(v_\star-\|(u_\star,v_\star)\|)^2}{\|(u_\star,v_\star)\|^2} 
    & 
      =\frac{\epsilon^2-\epsilon^4+(1-\epsilon^2)^2 -
      2(1-\epsilon^2)\sqrt{1-\epsilon^2}  
      +1-\epsilon^2}{1-\epsilon^2}
    \\
    & =\epsilon^2+1-\epsilon^2-2\sqrt{1-\epsilon^2}+1=2(1-\sqrt{1-\epsilon^2}),
  \end{align*}
  which completes the proof.
\end{proof}

\begin{proof}[Proof of
  Lemma~\ref{lemma:sampled:gradient:direction:bound}]
  Let $r>0$ and $x\in {\mathcal K_{\oplus r}^{\rm compl}}$. By continuity 
  of $\alpha$, there exists $\epsilon\in (0,r)$ such that $\alpha(y)\ge 
  2/3\alpha(x)$ for all $y\in B(x,\epsilon)$. In addition, by
  Assumption~\ref{assumptions:for:F}(ii), since $B(x,\epsilon)$
  is bounded, there exists a (uniform with
  respect to $\theta$) Lipschitz constant $L$ that bounds the
  gradients of $F$ on $B(x,\epsilon)\cap\mathcal D$. By
  shrinking $\epsilon$ further to guarantee that the right-hand side
  of \eqref{eta:increment:bound} is below $\alpha(x)/(3L)$, it
  follows from Lemma~\ref{lemma:normal:deviation:around:ball} and
  Assumption~\ref{assumption:gradient:cone} that
  \begin{align*}
    \langle \nabla_x F(y^i,\theta^i), \eta(x)\rangle
    &
      = \langle \nabla_x F(y^i,\theta^i), \eta(y^i)\rangle+\langle \nabla_x 
      F(y^i,\theta^i),\eta(x)-\eta(y^i)\rangle
    \\
    & \ge\alpha(y^i)-\|\nabla_x F(y^i,\theta^i)\| \|\eta(x)-\eta(y^i)\|
     {\ge}  2/3\alpha(x)-1/3\alpha(x)=1/3\alpha(x)
  \end{align*}
  for all $i$. Thus, for any
  $g\in \conv(\{\nabla_x F(y^1,\theta^1),\ldots, \nabla_x
  F(y^m,\theta^m)\})$ we obtain the bound in the statement.
\end{proof}

\begin{proof}[Proof of Proposition~\ref{prop:boundeness}]
  To prove the result, we need to show that whenever
  $x^k=(x_1^k,\ldots,\linebreak x_J^k)\in\prod_{j=1}^JB({\mathcal
    K_j}, \tau^\star+\eps)$, it also holds that
  $x^{k+1}=(x_1^{k+1},\ldots,x_J^{k+1})\in \prod_{j=1}^JB({\mathcal
    K_j},\tau^\star+\eps)$. It therefore suffices to show that for each
  $j=1,\ldots,J$, whenever $x_j^k\in B({\mathcal K_j},\tau^\star+\eps)$, it
  also holds that $x_j^{k+1}\in B({\mathcal K_j},\tau^\star+\eps)$. We
  distinguish two cases:
  
  \noindent \textit{Case (i): $x_j^k\in B({\mathcal K_j},\eps)$.} Then the 
  result follows directly by taking into account that the stepsize of the
  algorithm is bounded by $\tau\le\tau^\star$ and that
  $\|x_j^k-x_j^{k+1}\|\le\|x^k-x^{k+1}\|$.

  \noindent \textit{Case (ii): $x_j^k\in B({\mathcal K_j},\tau^\star+\eps)$ and
    $\dist(x_j^k,{\mathcal K_j})>\eps$.} Let $d^k=(d_1^k,\ldots,d_J^k)$ be 
    the descent direction of the algorithm and note that
  \begin{align*}
    d^k\in\conv(\{\nabla_xF(x^{ki},\theta_\star^{ki})\}_{i=1}^m)\quad {\rm and} 
    \quad \nabla_{x_j}F(x^{ki},\theta_\star^{ki})\in 
    \cone{x_j^{ki}-{\mathcal K_j}} 
  \end{align*}
  by Step 1 and Assumption~\ref{assumption:gradient:cone},
  respectively. Since each $x^{ki}\in B(x^k,\eps)$ and therefore also
  $x_j^{ki}\in B(x_j^k,\eps)$, we have that
  $\nabla_{x_j}F(x^{ki},\theta_\star^{ki})\in
  \cone{x_j^k-B({\mathcal K_j},\eps)}$ and it follows that also
  \begin{align} \label{dk:in:cone}
    d_j^k\in\cone{x_j^k-B({\mathcal K_j},\eps)}. 
  \end{align}
  If $d_j^k=0$, then $x_j^{k+1}=x_j^k+t^kd_j^k=x_j^k$ and the 
  proof is complete. Otherwise, let 
  \begin{align*}
  s:=\inf\{t>0\,|\,x_j^k+td_j^k\in B({\mathcal K_j},\eps)\}.
  \end{align*}
  By \eqref{dk:in:cone}, $s\in\RgO$. We distinguish two further subcases:
  
  \noindent \textit{Case (iia): $s\le t^k$, where $t^k$ is the
    {step size} of the algorithm at iteration $k$.}
  %
  %
  Then since $t^k\le \tau$, $\|d_j^k\|\le\|d^k\|=1$, and 
  $x_j^k+sd_j^k\in B({\mathcal K_j},\eps)$ by the definition of $s$, we 
  get  
  \begin{align*}
    d(x_j^{k+1},B({\mathcal K_j},\eps))=d(x_j^k+t^kd_j^k,B({\mathcal 
    K_j},\eps))\le\|x_j^k+t^kd_j^k-(x_j^k+sd_j^k)\|\le\tau\le\tau^\star,
  \end{align*}
  which establishes our claim.
  
  \noindent \textit{Case (iib): $s>t^k$.} Let
  $y:={\proj_{B({\mathcal K_j},\eps)}(x_j^k)}$ and denote 
  $z:=x_j^k+sd_j^k$. Since $x_j^k\in B({\mathcal K_j},\tau^\star+\eps)$, we have 
  that $\|y-x_j^k\|\le\tau^\star$. Now if we define
  \begin{align*}
    y'=\frac{t^k}{s}z+\frac{s-t^k}{s}y,
  \end{align*}   
  we get by convexity of $B({\mathcal K_j},\eps)$ and the fact that 
  $y,z\in B({\mathcal K_j},\eps)$, that also $y'\in B({\mathcal 
  K_j},\eps)$. It follows that 
  \begin{align*}
    \|x_j^{k+1}-y'\|
    & 
      =\Big\|x_j^k+t^kd^k-\Big(\frac{t^k}{s}z+\frac{s-t^k}{s}y\Big)\Big\|
    \\
    & =\Big\|x_j^k+t^kd^k-\frac{t^k}{s}(x_j^k+s d^k)-\frac{s-t^k}{s}y\Big\|
      =\frac{s-t^k}{s}\|x_j^k-y\|\le\frac{s-t^k}{s} \tau^\star\le \tau^\star,  
  \end{align*} 
  which establishes that $x_j^{k+1}\in B({\mathcal K_j},\tau^\star+\eps)$ and 
  concludes the proof.
\end{proof}

\subsection{Markov Chain generated by the mGS algorithm and zero-measure events} \label{appendix:MP}

Here we discuss how the mGS algorithm can be described in the language
of discrete-time stationary Markov processes and use this formalism to
prove that the events $\mathcal E_{\ell,k_1,\iota,k_2}^a$ and
$\mathcal E_{\ell,k_1,\iota,k_2}^b$ in the proof of
Theorem~\ref{thm:GS:convergence} have probability zero. To this end,
we need to define a suitable state space and a fixed-in-time
transition kernel capturing the evolution of the algorithm. We
consider the state space
$  \Xi:=\mathbb N_0\times\Rat{n}\times\Rat{mn}$, 
%
where each tuple $(\ell,z,(z^1,\ldots,z^m))\in\Xi$ represents
respectively the number $\ell-1\in\mathbb N_0$ of times the tolerances
$\eps_k$ and $\nu_k$ have been discounted, the iterate
$z\equiv x^k\in\Rat{n}$ of the algorithm, and the points
$z^1\equiv x^{k1},\ldots,z^m\equiv x^{km}$ sampled from
$D_{\eps_k}^m(x^k)\subset\Rat{mn}$. Since $\ell=1$ at the beginning of
the algorithm and $\ell$ does not decrease as long as the algorithm is
running, we reserve $\ell=0$ to render
${0}\times\Rat{n}\times\Rat{mn}\subset\Xi$ an absorbing set. The
algorithm only reaches this set in the event of probability zero where
it samples points outside $\mathcal D$.

To determine the Markov kernel generating the probability space of the algorithm, we first consider the measurable transition map $T:\Xi\times\Rat{mn}\to\Xi$ given for any $\xi=(\ell,z,(z^1,\ldots,z^m))\in\Xi$ and $u\in\Rat{mn}$ by
\begin{align*}
T(\xi,u)\equiv(T_1(\xi),T_2(\xi),T_3(\ell,z,u)).
\end{align*} 
Here $T_1(\xi)$ and $T_2(\xi)$ are the next tolerance discount exponent $\ell'$ and the next iterate $z'$ of the algorithm, respectively, based on the current discount exponent $\ell$, the current iterate $z\equiv x^k$, and the points $z^1\equiv x^{k1},\ldots,z^m\equiv x^{km}$ sampled around it. The last component is given by 
\begin{align*}
T_3(\xi,u):=(z'+\eps(\ell)u^1,\ldots,z'+\eps(\ell)u^m),\quad z'=T_2(\xi)
\end{align*}         
where $u\in\Rat{mn}$ represents a point that is uniformly sampled from $D_1^m(0)$. 

Based on this, we define the probability transition kernel
$P:\Xi\times\B(\Xi)\to [0,1]$, where $\B(\Xi)$ is the 
trace Borel $\sigma$ 
algebra on
$\Xi$, 
by setting
$P(\xi,A):=T_{\xi\#}\mu(A)$,
%
where $T_{\xi\#}\mu$ is the pushforward of the uniform probability
measure on $D_1^m(0)\subset\Rat{mn}$ through the map
\begin{align*}
T_\xi:\Rat{mn}\to\Xi,\quad u\mapsto T_\xi(u):=T(\xi,u). 
\end{align*}  
Clearly, $P(\xi,\cdot)$ is a probability measure on $(\Xi,\B(\Xi))$
for each $\xi\in\Xi$. In addition, for any $A\in\B(\Xi)$, the map
\begin{align*}
\xi\mapsto P(\xi,A)=T_{\xi\#}\mu(A)=\int_{\Rat{mn}}\bone_A(T_\xi(u))\mu(du)
=\int_{\Rat{mn}}\bone_A(T(\xi,u))\mu(du)
\end{align*} 
is measurable by Fubini's theorem and we conclude that $P$ satisfies
both requirements of a transition kernel.

\begin{rem}
\longthmtitle{Measurability of $T$}
{\rm Measurability of the map $T$ is ensured by measurability of the
  elementary operations carried along the steps of the algorithm. To
  this end, we only need to assume that the oracle which returns an
  approximate maximizer $\theta_\star$ for each $x\in\Rat{n}$ and
  accuracy $\delta>0$ is a measurable selection map with respect to
  $(x,\delta)\in \Rat{n}\times\RgO$. We also assume that the Lipschitz
  bounds $L_{\nabla_xF}^\theta(x)$, which are use to define the
  accuracy level of this selection in \eqref{Gk:dfn:second:part} are
  measurable with respect to $x$.}  \oprocend
\end{rem}

Next, we consider the event
\begin{align*}
  A:=\{\ell\}\times B(x_{\ell,\iota},2\widehat\tau(\ell,\iota))\times(\Rat{mn}\setminus U(\ell,\iota))
\end{align*} 
on $\B(\Xi)$, which includes all the possible states
$\xi^{k_2+1},\xi^{k_2+2},\ldots$ of the algorithm for the events in
$\mathcal E_{\ell,k_1,\iota,k_2}^a$. Namely, $x_{\ell,\iota}$,
$\widehat\tau(\ell,\iota)$, $U(\ell,\iota)$, and $k_2$ are selected
according to Case~(a) in Step~2 of the proof of
Theorem~\ref{thm:GS:convergence}. Then, denoting $\ell'=T_1(\xi)$ and
$z'=T_2(\xi)$ for each $\xi=(\ell,z,(z^1,\ldots,z^m))\in\Xi$, we have
\begin{align*}
  P(\xi,A) & =\int_{\Rat{mn}}\bone_A(T(\xi,u))\mu(du) \\
           & =\int_{\Rat{mn}}\bone_{\{\ell\}}(\ell')\bone_{B(x_{\ell,\iota},2\widehat\tau(\ell,\iota))}(z')\bone_{\Rat{mn}\setminus U(\ell,\iota)}(z'+\eps(\ell)u^1,\ldots,z'+\eps(\ell)u^m)\mu(du) \\
           & \le \int_{\Rat{mn}}\bone_{\Rat{mn}\setminus U(\ell,\iota)}(z'+\eps(\ell)u^1,\ldots,z'+\eps(\ell)u^m)\mu(du) \\
           & =\frac{{\rm Vol}_{mn}(D_{\eps(\ell)}^m(z'))-{\rm Vol}_{mn}(U(\ell,\iota))}{{\rm Vol}_{mn}(D_{\eps(\ell)}^m(z'))}\equiv\rho(\ell,\iota)<1,
\end{align*}   
which also yields
\begin{align} \label{kernekl:bound}
\int_AP(\zeta,d\xi)\le \rho(\ell,\iota)\;\textup{for all}\;\zeta\in A.
\end{align} 

Now let $Q$ denote the law of $\xi^{k_2}$ corresponding to iterate
$x^{k_2}$ of the algorithm. From \cite[Theorem 3.41, Page
60]{SM-RLT:09} the probability that $\xi^{k_2+j}\in A$ for all
$j\in[\nu]$, where $\nu\in\mathbb N$, is given by
\begin{align*}
\int_{\xi^{k_2}\in\Xi}\int_{\xi^{k_2+1}\in A}\cdots\int_{\xi^{k_2+\nu-1}\in A}
Q(d\xi^{k_2})P(\xi^{k_2},d\xi^{k_2+1})\cdots P(\xi^{k_2+\nu-1},A).
\end{align*} 
By \eqref{kernekl:bound}, this probability is bounded by
$\rho(\ell,\iota)^\nu$ and goes to zero as $\nu$ goes to infinity. We
therefore conclude that the event $\mathcal E_{\ell,k_1,\iota,k_2}^a$
has probability zero. The corresponding conclusion for
$\mathcal E_{\ell,k_1,\iota,k_2}^b$ is established in the same exact
way.

\bibliography{alias,JC,SM,SMD-add} 
\bibliographystyle{IEEEtranS}
\end{document}